\documentclass[10pt,a4paper]{amsart}
\usepackage{amstext}
\usepackage{amsthm}
\usepackage{amssymb}
\usepackage{stmaryrd}
\usepackage[unicode=true,pdfusetitle,
 bookmarks=true,bookmarksnumbered=false,bookmarksopen=false,
 breaklinks=false,pdfborder={0 0 1},backref=false,colorlinks=false]
 {hyperref}

\makeatletter
\numberwithin{equation}{section}
\numberwithin{figure}{section}
\theoremstyle{plain}
\newtheorem{thm}{\protect\theoremname}
  \theoremstyle{plain}
  \newtheorem{lem}[thm]{\protect\lemmaname}
  \theoremstyle{remark}
  \newtheorem{rem}[thm]{\protect\remarkname}
  \theoremstyle{definition}
  \newtheorem{defn}[thm]{\protect\definitionname}
  \theoremstyle{plain}
  \newtheorem{prop}[thm]{\protect\propositionname}

\newcounter{cprop}[section]
\numberwithin{equation}{section}
\numberwithin{cprop}{section}
\ifx\c@thm\undefined
 \newcounter{thm}
\fi
\numberwithin{thm}{section}

\newcommand{\mcB}{\mathcal{B}}

\newcommand{\E}{\mathbb{E}}

\renewcommand{\d}{\delta}
\newcommand{\z}{\zeta}
\renewcommand{\and}{\quad\textrm{ and }\quad}

\renewcommand{\P}{\mathbb{P}}
\renewcommand{\a}{\alpha}
\renewcommand{\b}{\beta}

\renewcommand{\o}{\omega}
\renewcommand{\O}{\Omega}
\renewcommand{\t}{\theta}

\renewcommand{\div}{\textnormal{div}}

\newcommand{\supp}{\;\textrm{supp}\;}

\newcommand{\mcT}{\mathcal{T}}

\newcommand{\mcG}{\mathcal{G}}

\newcommand{\mcO}{\mathcal{O}}

\newcommand{\mcF}{\mathcal{F}}

\newcommand{\mcM}{\mathcal{M}}
\newcommand{\F}{\mathcal F}

\newcommand{\mcH}{\mathcal H}

\renewcommand{\>}{\rangle}

\newcommand{\D}{\Delta}
\newcommand{\g}{\gamma}

\newcommand{\vp}{\varphi}

\newcommand{\ve}{\varepsilon}

\newcommand{\norm}[1]{\left\lVert#1\right\rVert}

\newcommand{\abs}[1]{\lvert#1\rvert}
\newcommand{\lrabs}[1]{\left\lvert#1\right\rvert}

\newcommand{\Bcal}{{\mathcal B}}

\newcommand{\Kcal}{{\mathcal K}}

\newcommand{\Ocal}{{\mathcal O}}

\renewcommand{\a}{\alpha}
\renewcommand{\b}{\beta}
\newcommand{\eps}{{\varepsilon}}

\newcommand{\R}{\mathbbm{R}}
\newcommand{\N}{\mathbbm{N}}

\renewcommand{\P}{\mathbbm{P}}

\newcommand{\esssup}{\operatorname{ess\;sup}}

\renewcommand{\div}{\operatorname{div}}

\newcommand{\Lip}{\operatorname{Lip}}

\newcommand{\sgn}{\operatorname{sgn}}
\renewcommand{\le}{\leq}
\renewcommand{\ge}{\geq}

\newcommand{\BIGOP}[1]{\mathop{\mathchoice%
{\raise-0.22em\hbox{\huge $#1$}}%
{\raise-0.05em\hbox{\Large $#1$}}{\hbox{\large $#1$}}{#1}}}

\newcommand{\BIGboxplus}{\mathop{\mathchoice%
{\raise-0.35em\hbox{\huge $\boxplus$}}%
{\raise-0.15em\hbox{\Large $\boxplus$}}{\hbox{\large $\boxplus$}}{\boxplus}}}

\def\leq {\leqslant}
\def\geq {\geqslant}

\usepackage{amssymb,amsthm,amsmath,amsfonts,amscd}
\usepackage{amsaddr}
\usepackage{graphicx}
\usepackage{url}
\usepackage{color}
\usepackage[active]{srcltx}
\usepackage[matrix,arrow]{xy}
\usepackage{mathrsfs}
\usepackage{enumerate}
\usepackage{amsopn} 
\usepackage{bbm} 
\usepackage{cite}
\usepackage{hyperref}
\allowdisplaybreaks[1]
\usepackage[english]{babel}
\usepackage{bbm}
\usepackage{mhchem}

\date{\today}


\makeatother

  \providecommand{\definitionname}{Definition}
  \providecommand{\lemmaname}{Lemma}
  \providecommand{\propositionname}{Proposition}
  \providecommand{\remarkname}{Remark}
\providecommand{\theoremname}{Theorem}

\begin{document}

\title[Ergodicity and local limits for stochastic $p$-Laplace equations]{Ergodicity and local limits for stochastic local and nonlocal \texorpdfstring{\MakeLowercase{$p$}}{\MakeLowercase{p}}-Laplace
equations}
\begin{abstract}
Ergodicity for local and nonlocal stochastic singular $p$-Laplace
equations is proven, without restriction on the spatial dimension
and for all $p\in[1,2)$. This generalizes previous results from {[}Gess,
T\"{o}lle; J. Math. Pures Appl., 2014{]}, {[}Liu, T\"{o}lle; Electron.
Commun. Probab., 2011{]}, {[}Liu; J. Evol. Equations, 2009{]}. In
particular, the results include the multivalued case of the stochastic
(nonlocal) total variation flow, which solves an open problem raised
in {[}Barbu, Da Prato, R\"{o}ckner; SIAM J. Math. Anal., 2009{]}.
Moreover, under appropriate rescaling, the convergence of the unique
invariant measure for the nonlocal stochastic $p$-Laplace equation
to the unique invariant measure of the local stochastic $p$-Laplace
equation is proven.
\end{abstract}

\author{Benjamin Gess}

\address{Max-Planck Institute for Mathematics in the Sciences\\
Inselstra\ss e 22\\
04103 Leipzig\\
Germany}

\email{bgess@mis.mpg.de}

\author{Jonas M. T\"{o}lle }

\address{Aalto University\\
School of Science\\
Department of Mathematics and Systems Analysis\\
PO Box 11100\\
FI-00076 Aalto\\
Finland}

\email{jonas.tolle@aalto.fi}


\thanks{J.M.T. gratefully acknowledges travel funds granted by the CRC 701
``Spectral Structures and Topological Methods in Mathematics'' of
the DFG (German Research Foundation)}
\maketitle

\section{Introduction}

We consider stochastic nonlocal singular $p$-Laplace equations of
the type 
\begin{align}
dX_{t} & \in\left(\int_{\mcO}J(\cdot-\xi)|X_{t}(\xi)-X_{t}(\cdot)|^{p-2}(X_{t}(\xi)-X_{t}(\cdot))\,d\xi\right)dt+BdW_{t}\label{eq:intro_nonlocal_plp}\\
X_{0} & =x_{0},\nonumber 
\end{align}
and stochastic (local) singular $p$-Laplace equations of the type
\begin{align}
dX_{t} & \in\div\left(|\nabla X_{t}|^{p-2}\nabla X_{t}\right)\,dt+BdW_{t}\label{eq:intro_plp}\\
X_{0} & =x_{0},\nonumber 
\end{align}
with zero Neumann boundary conditions on bounded, convex domains $\mcO\subseteq\R^{d}$
with smooth boundary $\partial\mcO$, mean zero initial conditions
$x_{0}\in H:=L_{\operatorname{av}}^{2}(\Ocal)$ and $p\in[1,2)$.
Here, $W$ is a cylindrical Wiener process on $H$, $B\in L_{2}(H)$
is a symmetric Hilbert-Schmidt operator and $J:\R^{d}\to\R$ is a
nonnegative, continuous, radial kernel with compact support and $J(0)>0$.
In particular, this includes the multi-valued case of the stochastic
total variation flow ($p=1$) recently studied in \cite{BDPR09-4,BR13}.
We note that for $p=1$ the equations \eqref{eq:intro_nonlocal_plp}
and \eqref{eq:intro_plp} become evolution inclusions.

Our results are twofold: First, we prove the existence and uniqueness
of an invariant probability measure to \eqref{eq:intro_nonlocal_plp}
and \eqref{eq:intro_plp}. Second, the convergence of the respective
invariant probability measures for \eqref{eq:intro_nonlocal_plp}
to the invariant probability measure for \eqref{eq:intro_plp} is
shown, under appropriate rescaling of the kernel $J$.

Uniqueness of invariant probability measures to \eqref{eq:intro_plp}
has been previously considered in \cite{GT11,LT11,L09-1}. The difficulties
arising in proving uniqueness of invariant probability measures for
\eqref{eq:intro_plp} are due to the singular nature of the drift
and the resulting low regularity properties of the solutions. More
precisely, the energy space associated to \eqref{eq:intro_plp} is
given by $W_{\operatorname{av}}^{1,p}$, which is compactly embedded
into $L_{\operatorname{av}}^{2}$ only if 
\begin{equation}
d<\frac{2p}{2-p}.\label{eq:intro_dim_restr}
\end{equation}
The validity of this embedding is crucial for previously established
methods and thus \eqref{eq:intro_dim_restr} had to be assumed in
all of the works \cite{GT11,LT11,L09-1}, which led to stringent restrictions
on the spatial dimension $d$, e.g.~$d\le2$ for $p\approx1$. For
the case of nonlocal stochastic $p$-Laplace equations, the situation
is even worse, since the energy associated to \eqref{eq:intro_nonlocal_plp}
is given by\footnote{\label{fn:1}For the sake of notational simplicity we drop normalization
constants in the introduction.} 
\[
\vp(u)=\frac{1}{2p}\int_{\mcO}\int_{\mcO}J(\z-\xi)|u(\xi)-u(\z)|^{p}\,d\xi d\z
\]
which is equivalent to the $L^{p}$ norm. Hence, based on this no
compactness and thus tightness for the laws of the solutions in $L_{\operatorname{av}}^{2}$
can be expected. These obstacles are overcome in the present work,
by establishing a cascade of energy inequalities for $L^{m}$ norms
of the solutions to \eqref{eq:intro_nonlocal_plp} and \eqref{eq:intro_plp}
for all $m\ge2$. These new estimates are then used in order to prove
concentration of mass of the solutions around zero, which in turn
allows the application of results developed in \cite{KPS10}, based
on coupling techniques. In conclusion, we prove the existence and
uniqueness of an invariant probability measure for \eqref{eq:intro_nonlocal_plp}
and \eqref{eq:intro_plp} without any restriction on the dimension
$d\in\N$ and for all $p\in[1,2)$. In particular, this solves the
open question raised in \cite{BDPR09-4} of uniqueness of invariant
measures for the stochastic total variation flow. 

In the second part of this paper, we consider the convergence of invariant
probability measures under rescaling of the kernel $J$ in \eqref{eq:intro_nonlocal_plp}.
More precisely, we consider
\begin{align}
dX_{t}^{\ve} & =\left(\int_{\mcO}J^{\ve}\left(\cdot-\xi\right)|X_{t}^{\ve}(\xi)-X_{t}^{\ve}(\cdot)|^{p-2}(X_{t}^{\ve}(\xi)-X_{t}^{\ve}(\cdot))d\xi\right)dt+BdW_{t},\label{eq:intro_nonlocal_conv}
\end{align}
where $p\in(1,2)$ and\textsuperscript{\ref{fn:1}}
\[
J^{\ve}(\xi)=\frac{1}{\ve^{d+p}}J\left(\frac{\xi}{\ve}\right),\quad\xi\in\R^{d}
\]
and prove that the corresponding invariant measures $\mu^{\ve}$ converge
weakly$^{*}$ to the invariant measure $\mu$ corresponding to \eqref{eq:intro_plp}.
Somewhat related questions of convergence of invariant measures of
\eqref{eq:intro_plp} with respect to perturbations in $p$ have been
considered in \cite{CiotToe2,CT12}, under stringent restrictions
on the spatial dimension, i.e.~assuming \eqref{eq:intro_dim_restr}.
Again, such dimensional restrictions are crucial to the approach developed
in \cite{CiotToe2,CT12}, since the argument relies on tightness of
the respective sequence of invariant probability measures $\mu^{p}$,
which in turn is verified using the compactness of the embedding $W^{1,p}\hookrightarrow L^{2}$. 

In the setting of local limits for \eqref{eq:intro_nonlocal_conv}
for general dimension $d$, this leads to two fundamental problems:
First, no concentration of the invariant probability measures on some
uniform, compactly embedded space can be expected. Second, as observed
in \cite{GT15}, only weak convergence of the solutions to \eqref{eq:intro_nonlocal_plp}
to the solution to \eqref{eq:intro_plp} is available, that is, $X_{t}^{\ve}\rightharpoonup X_{t}$
in $H$ for $\ve\to0$. Hence, we do not have the convergence of the
associated Markovian semigroups $P_{t}^{\ve}F$ for all $F\in\Lip_{b}(H)$,
a crucial ingredient in previously developed methods such as in \cite{CiotToe2,CT12}.
These problems are overcome in the present work and we prove that
$\mu^{\ve}$ converges to $\mu$ in the topology of weak$^{*}$ convergence
of measures on $L_{\operatorname{av}}^{p}$, without any restriction
on the spatial dimension $d$.

We note that, in general, the invariant measures $\mu^{\ve}$ to \eqref{eq:intro_nonlocal_conv}
will only be concentrated on the domains of the corresponding energy
functionals\textsuperscript{\ref{fn:1}}
\[
\vp^{\ve}(u):=\frac{1}{2p}\int_{\Ocal}\int_{\Ocal}J^{\ve}\left(\zeta-\xi\right)\lrabs{u(\xi)-u(\zeta)}^{p}\,d\xi d\zeta,
\]
rather than on $W_{\operatorname{av}}^{1,p}$ as for \eqref{eq:intro_plp}.
Roughly speaking, one has $p\vp^{\ve}(u)\uparrow\|u\|_{W^{1,p}}^{p}$.
In this sense, at least \emph{asymptotic} concentration on $W_{\operatorname{av}}^{1,p}$
is still satisfied. This is reflected in our proof by working with
\emph{asymptotic} tightness rather than tightness. Non-compactness
of $W_{\operatorname{av}}^{1,p}$ in $L_{\operatorname{av}}^{2}$
is dealt with by considering weak$^{*}$ convergence of measures on
$L_{\operatorname{av}}^{p}$ rather than on $L_{\operatorname{av}}^{2}$.
However, this leads to the further difficulty of working with two
topologies: weak$^{*}$ convergence of $\mu^{\ve}$ on $L_{\operatorname{av}}^{p}$
and weak convergence of $X_{t}^{\ve}$ on $L_{\operatorname{av}}^{2}$.
These issues are resolved by a careful treatment in Section \ref{secc:convergence_measures}
below.

For simplicity, we restrict to the case of zero Neumann boundary conditions.
In the nonlocal form \eqref{eq:intro_nonlocal_plp} the choice of
zero Neumann boundary conditions is reflected by the choice of the
domain of integration as $\Ocal$, rather than, for example, $\Ocal+\operatorname{supp}J$.
Under appropriate rescaling of $J$ it is known (cf.~\cite{AVMRT10})
that the solutions to the nonlocal deterministic equations, that is
\eqref{eq:intro_nonlocal_plp} with $B\equiv0$, converge to the solution
of the local $p$-Laplace equation with zero Neumann boundary conditions,
that is, to \eqref{eq:intro_plp} with $B\equiv0$. The nonlocal analogue
to homogeneous Dirichlet boundary conditions involves a penalizing
term (cf.~\cite{AMRT09-2,AVMRT10}) which can be viewed as a nonlocal
analogue of the boundary trace. While we focus on Neumann boundary
conditions, we expect that the case of Dirichlet boundary conditions
can be treated by similar methods.

The case of a degenerate drift, that is, $p\ge2$ in \eqref{eq:intro_nonlocal_plp},
\eqref{eq:intro_plp} can be treated by rather different and somewhat
more simple methods. More precisely, for $p\ge2$ the dissipativity
method (cf.~e.g.~\cite[Theorem 3.7]{G13}) can be applied to obtain
the ergodicity and strong mixing property for both \eqref{eq:intro_nonlocal_plp}
and \eqref{eq:intro_plp}. Concerning the convergence of the invariant
measures, in contrast to the singular case $p<2$, in the degenerate
case the embedding $W_{\operatorname{av}}^{1,p}\hookrightarrow L_{\operatorname{av}}^{2}$
is always compact. Hence, this compactness may be used to deduce (asymptotic)
tightness of the invariant measures without any restriction on the
dimension, thus allowing for a rather direct argument similar to the
one given in \cite{CT12}.

The problems of existence, uniqueness and stability with respect to
parameters of invariant measures for SPDE are classical and a review
of the available results would exceed the scope of this paper. Thus,
we shall restrict to mention some exemplary works in this direction
and the references therein. A typical approach to the uniqueness of
invariant measures is given by the Doob-Khasminskii Theorem \cite{D48,DPZ96,H60}
and its more recent generalization \cite{HM06,HM08}. In both cases,
this strategy requires smoothing properties of the associated Markov
semigroup and its irreducibility, which have been successfully verified
for many semilinear SPDE with degenerate noise (e.g. \cite{HM08,FGRT15,Romito:2011fc,Albeverio:2012fh}
and the references therein). The route followed in this paper is different
and relies on a contractivity ($e$-property) of the Markov semigroup,
rather than on a smoothing property (asymptotic strong Feller property),
as suggested in the abstract framework of \cite{KPS10}. Some details
on the relation of the asymptotic strong Feller property and the $e$-property
can be found in \cite{Kapica:2011hj,SzWo12,J13}. We note that this
type of argument shows resemblance to arguments used to prove ergodicity
of stochastic scalar conservation laws \cite{DV10,B16,S16}. Concerning
the stability of invariant measures for stochastic Navier-Stokes and
stochastic Burgers equations with respect to parameters we refer to
\cite{K04,Shirikyan:2007ch,Kuksin:2008go} and the references therein.

A detailed treatment of deterministic nonlocal $p$-Laplace equations
may be found in \cite{AMRT08,AMRT09,AVMRT10,AMRT11} and the references
therein. Decay estimates and extinction results for solutions of deterministic
nonlocal $p$-Laplace equations have been considered e.g. in \cite{BelaudDiaz2010,FerreiraRossi2015,HerreroVazquez1981,IgnatPinascoRossiAntolin2014,IgnatRossi2009,Porzio2009,Porzio2011}.
Relying on non-degeneracy assumptions on the noise, gradient estimates,
Harnack inequalities and exponential convergence rates for stochastic
$p$-Laplace equations and stochastic porous media equations have
been obtained in \cite{W15,W15-2,L09-1} and the references therein.
A stochastic variational inequality (SVI) approach to stochastic fast
diffusion equations has been developed in \cite{GR15} and their ergodicity
has been considered in \cite{BDP10,LW08,Liu11,GT11,W15,W15-2,LT11}.
The case of stochastic degenerate $p$-Laplace equations, that is
for $p>2$, has been investigated in \cite{MZ10,BDP06,PR07,L09,W13,W15,W15-2,GessRoe2014}
and ergodicity for stochastic porous media equations has been obtained
in \cite{GZ14,K06,O11,W07,BBDPR06,DPRRW06,BDP06,PR07,W13,L09,DPR04-2}.

\subsection{Structure of the paper}

In Section \ref{sec:ergodicity_nonlocal} ergodicity for the stochastic
nonlocal $p$-Laplace equation is proven. The case of the stochastic
local $p$-Laplace equation is treated in Section \ref{sec:ergodicity_local}.
Convergence of the solutions of the nonlocal stochastic $p$-Laplace
equation to its local version is shown in Section \ref{sec:to_local_sol}.
The respective convergence of invariant probability measures is shown
in Section \ref{secc:convergence_measures}. For notations see Appendix
\ref{sec:Notation}.

\section{Ergodicity for stochastic nonlocal \texorpdfstring{$p$}{p}-Laplace
equations\label{sec:ergodicity_nonlocal}}

In this section we derive a stochastic variational inequality (SVI)
formulation for stochastic singular nonlocal $p$-Laplace equations
with homogeneous Neumann boundary condition of the type
\begin{align}
dX_{t} & \in\left(\int_{\mcO}J(\cdot-\xi)|X_{t}(\xi)-X_{t}(\cdot)|^{p-2}(X_{t}(\xi)-X_{t}(\cdot))d\xi\right)dt+BdW_{t}\label{eq:nonlocal_plp-1-1}\\
X_{0} & =x_{0}\in L^{2}(\O,\mcF_{0};L_{\operatorname{av}}^{2}(\mcO)),\nonumber 
\end{align}
where $p\in[1,2)$ and $\mcO$ is a bounded, smooth domain in $\R^{d}$.
The kernel $J:\R^{d}\to\R$ is supposed to be a nonnegative, continuous,
radial function with compact support, $J(0)>0$ and $\int_{\R^{d}}J(z)\,dz=1$.
Furthermore, $W$ is a cylindrical Wiener process on $H$ and $B\in L_{2}(H)$
symmetric with $H=L_{\operatorname{av}}^{2}(\mcO)$. Hence, 
\[
W_{t}^{B}:=BW_{t}
\]
is a trace-class Wiener process in $H$. We further assume that there
is an orthonormal basis $e_{k}$ of $H$ such that 
\begin{equation}
\sum_{k=1}^{\infty}\|Be_{k}\|_{\infty}^{2}<\infty,\label{eq:noise_bound}
\end{equation}
cf.~e.g.~\cite{BDPR09-4} where similar conditions on \textbf{$B$
}have been used in the case of the stochastic total variation flow.
For $u\in L^{p}(\mcO)$ we set 
\[
\vp(u):=\frac{1}{2p}\int_{\mcO}\int_{\mcO}J(\z-\xi)|u(\xi)-u(\z)|^{p}\,d\xi d\z
\]
and obtain, if $p>1$, 
\[
A(u):=-\partial_{L^{2}}\vp(u)=\int_{\mcO}J(\cdot-\xi)|u(\xi)-u(\cdot)|^{p-2}(u(\xi)-u(\cdot))\,d\xi
\]
and, if $p=1$, 
\begin{align*}
A(u):= & -\partial_{L^{2}}\vp(u)\\
= & \Big\{\int_{\mcO}J(\cdot-\xi)\eta(\xi,\cdot)\,d\xi:\,\|\eta\|_{L^{\infty}}\le1,\,\eta(\xi,\z)=-\eta(\z,\xi)\text{ and }\\
 & \hskip20ptJ(\z-\xi)\eta(\xi,\z)\in J(\z-\xi)\sgn(u(\xi)-u(\z))\text{ for a.e. }(\xi,\z)\in\mcO\times\mcO\Big\},
\end{align*}
where $\partial_{L^{2}}\vp$ denotes the $L^{2}$ subgradient of $\vp$
restricted to $L^{2}$. We note that $A$ defines a continuous, monotone
operator on $H$, satisfying 
\begin{equation}
\|A(u)\|_{H}^{2}\lesssim1+\|u\|_{H}^{2}\quad\forall u\in H.\label{eq:A_growth}
\end{equation}
Hence, we can write \eqref{eq:nonlocal_plp-1-1} in its relaxed form
\[
dX_{t}\in-\partial\vp(X_{t})dt+BdW_{t}.
\]
Existence and uniqueness of an SVI solution $X=X^{x_{0}}\in L^{2}(\O;C([0,T];L^{2}(\mcO)))$
to \eqref{eq:nonlocal_plp-1-1} has been proven in \cite[Section 4]{GT15}
and
\begin{equation}
\E\|X_{t}^{x}-X_{t}^{y}\|_{L^{2}}^{2}\lesssim\|x-y\|_{L^{2}}^{2}.\label{eq:l2-contraction-2}
\end{equation}
Since $\int_{\mcO}\eta d\z=0$ for all $\eta\in A(u)$, $u\in H$,
from the construction of SVI solutions (see \cite[Definition 2.1]{GT15}
for the definition) presented in \cite[Section 4]{GT15} it easily
follows that the average value is preserved, that is, 
\begin{equation}
X_{t}\in L_{\operatorname{av}}^{2}(\mcO)\quad\forall t\ge0,\P\text{-a.s.}\label{eq:av-zero}
\end{equation}
if $x_{0}\in L^{2}(\O;L_{\operatorname{av}}^{2}(\mcO))$. Furthermore,
analogously to \cite[proof of Proposition 5.2]{GT11}, it follows
that 
\[
P_{t}F(x):=\E F(X_{t}^{x})\quad\text{for }F\in\mcB_{b}(H)
\]
defines a Feller semigroup on $\mcB_{b}(H)$.  As a main result in
this Section we obtain
\begin{thm}
\label{thm:nonlocal_ergodicity}There is a unique invariant measure
$\mu$ for $P_{t}$ satisfying\footnote{Cf.~Appendix \ref{sec:Notation} for notation.\label{fn:2}}
\[
0\in\operatorname{supp}(\mu)\subseteq\mcT(\mu)
\]
and 
\begin{equation}
\int_{H}\vp(x)d\mu(x)\lesssim\|B\|_{L_{^{2}}(H)}^{2}.\label{eq:ipm_mass_bound}
\end{equation}
\end{thm}
We first need to derive suitable a-priori bounds on general $L^{m}$
norms of the solutions. 
\begin{lem}
\label{lem:general_lp_bound}Let $x_{0}\in L^{m}(\O,\mcF_{0};L_{\operatorname{av}}^{m}(\mcO))$,
$m\in[2,\infty)$ and let $X$ be the corresponding SVI solution to
\eqref{eq:nonlocal_plp-1-1}. Then there are $c=c(m),C=C(m)>0$ such
that
\begin{align}
\frac{1}{m}\E\|X_{t}\|_{m}^{m}+c\E\int_{0}^{t}\|X_{r}\|_{p+m-2}^{p+m-2}dr & \le\frac{1}{m}\E\|x_{0}\|_{m}^{m}+tC\quad\forall t\ge0.\label{eq:nonlocal_energy_bound}
\end{align}
If $B\equiv0$ then we can choose $C=0$. 
\end{lem}
\begin{proof}
For notational convenience let 
\begin{align*}
\psi(\xi):= & \frac{1}{p}|\xi|^{p}\\
\phi(\xi):= & \partial\psi(\xi)=|\xi|^{p-2}\xi,\quad\xi\in\R^{d}.
\end{align*}
\textit{Step 1:} We start by proving that for $x\in L^{m}(\O,\mcF_{0};L_{\operatorname{av}}^{m}(\mcO))$,
we have $\E\|X_{t}\|_{L^{m}}^{m}<\infty$ for all $t\ge0.$ 

We aim to apply It\^{o}'s formula for $\frac{1}{m}\|\cdot\|_{m}^{m}$.
To do so, we need to consider appropriate approximations. Let 
\[
\iota^{\a}(r):=\frac{1}{m}\begin{cases}
|r|^{m} & \text{if }|r|\le\frac{1}{\a}\\
|\frac{1}{\a}|^{m}+\frac{m}{\a^{m-1}}(r-\frac{1}{\a})+\frac{m(m-1)}{2\a{}^{m-2}}(r-\frac{1}{\a})^{2} & \text{if }r\ge\frac{1}{\a}\\
|\frac{1}{\a}|^{m}+\frac{m}{\a^{m-1}}(r+\frac{1}{\a})+\frac{m(m-1)}{2\a{}^{m-2}}(r+\frac{1}{\a})^{2} & \text{if }r\le-\frac{1}{\a}.
\end{cases}
\]
and observe, for $\a$ small enough,
\begin{align}
(\iota^{\a})''(r): & =\begin{cases}
(m-1)|r|^{m-2} & \text{if }|r|\le\frac{1}{\a}\\
\frac{m-1}{|\a|^{m-2}} & \text{otherwise. }
\end{cases}\label{eq:iota_der}\\
 & \lesssim1+\iota^{\a}(r).\nonumber 
\end{align}
Let $\t^{\b}$ be a standard Dirac sequence on $\R^{d}$. For $v\in L^{2}(\mcO)$
we set\textsuperscript{\ref{fn:2}}
\begin{align*}
\eta^{\a}(v): & =\int_{\mcO}\iota^{\a}\left(v\right)d\z\\
\eta^{\a,\b}(v): & =\int_{\mcO}\iota^{\a}\left(\t^{\b}\ast\bar{v}\right)d\z
\end{align*}
and observe that $\eta^{\a,\b}\in C^{2}(L^{2})$ with uniformly continuous
derivatives on bounded sets. We recall that the SVI solution $X$
to \eqref{eq:nonlocal_plp-1-1} has been constructed in \cite[Section 4]{GT15}
as a limit in $L^{2}(\O;C([0,T];H))$ of (strong) solutions $X^{\d}$
corresponding to the approximating SPDE
\begin{align*}
dX_{t}^{\d} & =\left(\int_{\mcO}J(\cdot-\xi)\phi^{\d}(X_{t}^{\d}(\xi)-X_{t}^{\d}(\cdot))d\xi\right)dt+BdW_{t},
\end{align*}
where $\psi^{\d}$ is the Moreau-Yosida approximation (cf.~e.g.~\cite{B93})
of $\psi(\cdot)=\frac{1}{p}|\cdot|^{p}$ and $\phi^{\d}:=\partial\psi^{\d}$.
Hence, by It\^{o}'s formula
\begin{align}
\E\eta^{\a,\b}(X_{t}^{\d})= & \E\eta^{\a,\b}(x_{0})+\E\int_{0}^{t}\int_{\mcO}(\iota^{\a})'(\t^{\b}\ast X_{r}^{\d})(\t^{\b}\ast A^{\d}(X_{r}^{\d}))d\z dr\label{eq:approx}\\
 & +\sum_{k=1}^{\infty}\E\int_{0}^{t}\int_{\mcO}(\iota^{\a})''(\t^{\b}\ast X_{r}^{\d})(\t^{\b}\ast Be_{k})^{2}d\z dr.\nonumber 
\end{align}
Using (for $\a>0$ fixed) 
\begin{align*}
(\iota^{\a})'(r) & \lesssim1+|r|\\
(\iota^{\a})''(r) & \lesssim1\quad\forall r\in\R,
\end{align*}
and dominated convergence, we may let $\b\to0$ in \eqref{eq:approx}
to obtain that
\begin{align}
\E\eta^{\a}(X_{t}^{\d})\le & \E\eta^{\a}(x_{0})+\E\int_{0}^{t}\int_{\mcO}(\iota^{\a})'(X_{r}^{\d})A^{\d}(X_{r}^{\d})d\z dr\label{eq:approx_2}\\
 & +\sum_{k=1}^{\infty}\E\int_{0}^{t}\int_{\mcO}(\iota^{\a})''(X_{r}^{\d})(Be_{k})^{2}d\z dr.\nonumber 
\end{align}
We note that using \cite[Lemma 6.5]{AVMRT10}, monotonicity of $(\iota^{\a})'$
and $\sgn(\phi^{\d}(a-b))=\sgn(a-b)$ we have that
\begin{align*}
 & \int_{\mcO}(\iota^{\a})'(v)A^{\d}(v)d\z\\
 & =\int_{\mcO}\int_{\mcO}J(\z-\xi)\phi^{\d}(v(\xi)-v(\z))(\iota^{\a})'(v(\z))d\xi d\z\\
 & =-\frac{1}{2}\int_{\mcO}\int_{\mcO}J(\z-\xi)\phi^{\d}(v(\xi)-v(\z))((\iota^{\a})'(v(\xi))-(\iota^{\a})'(v(\z)))d\xi d\z\\
 & \le0,
\end{align*}
for all $v\in H$. Hence, using \eqref{eq:iota_der} we observe that
\begin{align*}
\E\eta^{\a}(X_{t}^{\d})\le & \E\eta^{\a}(x_{0})+C\E\int_{0}^{t}(1+\eta^{\a}(X_{r}^{\d}))dr.
\end{align*}
Gronwall's Lemma then implies that 
\begin{align*}
\E\eta^{\a}(X_{t}^{\d}) & \lesssim\E\eta^{\a}(x_{0})+1\\
 & \le\frac{1}{m}\E\|x_{0}\|_{m}^{m}+1.
\end{align*}
Hence, taking $\a\to0$ and using Fatou's Lemma we obtain that
\begin{align}
\frac{1}{m}\E\|X_{t}^{\d}\|_{m}^{m} & \lesssim\frac{1}{m}\E\|x_{0}\|_{m}^{m}+1.\label{eq:Xd_Lm_bound}
\end{align}
Taking $\d\to0$ finishes the proof.

\textit{Step 2: }We first note that it is enough to prove \eqref{eq:nonlocal_energy_bound}
for $x\in L^{\infty}(\O,\mcF_{0};L_{\operatorname{av}}^{\infty}(\mcO))$.
Due to \eqref{eq:l2-contraction-2} the case of $x\in L^{m}(\O,\mcF_{0};L_{\operatorname{av}}^{m}(\mcO))$
can then be concluded by approximation and Fatou's Lemma. Hence, assume
$x\in L^{\infty}(\O,\mcF_{0};L_{\operatorname{av}}^{\infty}(\mcO))$
from now on. By step one we have $\E\|X_{t}\|_{m}^{m}<\infty$ for
all $t\ge0$, $m\in\N$.

Letting $\a\to0$ in \eqref{eq:approx_2}, using dominated convergence
and \eqref{eq:A_growth}, we obtain
\begin{align*}
\frac{1}{m}\E\|X_{t}^{\d}\|_{m}^{m}\le\frac{1}{m} & \E\|x_{0}\|_{m}^{m}+\E\int_{0}^{t}\int_{\mcO}(X_{r}^{\d})^{[m-1]}A^{\d}(X_{r}^{\d})d\z dr\\
 & +(m-1)\sum_{k=1}^{\infty}\E\int_{0}^{t}\int_{\mcO}|X_{r}^{\d}|^{m-2}(Be_{k})^{2}d\z dr.
\end{align*}
Using \cite[Lemma 6.5]{AVMRT10} we observe that
\begin{align*}
 & \int_{\mcO}(X_{r}^{\d})^{[m-1]}A^{\d}(X_{r}^{\d})d\z\\
 & =\int_{\mcO}\int_{\mcO}J(\z-\xi)\phi^{\d}(X_{r}^{\d}(\xi)-X_{r}^{\d}(\z))(X_{r}^{\d})^{[m-1]}(\z)d\xi d\z\\
 & =-\frac{1}{2}\int_{\mcO}\int_{\mcO}J(\z-\xi)\phi^{\d}(X_{r}^{\d}(\xi)-X_{r}^{\d}(\z))((X_{r}^{\d})^{[m-1]}(\xi)-(X_{r}^{\d})^{[m-1]}(\z))d\xi d\z.
\end{align*}
From \cite{DPRRW06}, for every $m\in[2,\infty)$ there is a $c>0$
such that
\[
(a-b)(a^{[m-1]}-b^{[m-1]})\ge c|a-b|^{m}\quad\forall a,b\in\R.
\]
Since $\sgn(\phi^{\d}(a-b))=\sgn(a-b)$ and (cf.~\cite[Appendix A]{GT15})
\begin{align*}
\phi^{\d}(a)a & \ge\psi^{\d}(a)\\
 & \ge c\psi(a)-C\d
\end{align*}
for some $c>0$, $C>0$ and $\d>0$ small enough, this yields
\begin{align*}
\phi^{\d}(a-b)(a^{[m-1]}-b^{[m-1]}) & \ge c\phi^{\d}(a-b)(a-b)|a-b|^{m-2}\\
 & \ge c\psi^{\d}(a-b)|a-b|^{m-2}\\
 & \ge c|a-b|^{p+m-2}-C\d|a-b|^{m-2}\quad\forall a,b\in\R.
\end{align*}
Using this and the Poincar\'{e} type inequality \cite[Proposition 6.19]{AVMRT10}
in combination with \eqref{eq:av-zero} we get
\begin{align*}
 & \E\int_{\mcO}(X_{r}^{\d})^{[m-1]}A^{\d}(X_{r}^{\d})d\z\\
 & \le-\E\int_{\mcO}\int_{\mcO}J(\z-\xi)(c|X_{r}^{\d}(\xi)-X_{r}^{\d}(\z)|^{p+m-2}-C\d|X_{r}^{\d}(\xi)-X_{r}^{\d}(\z)|^{m-2})d\z d\xi\\
 & \le-c\E\|X_{r}^{\d}\|_{p+m-2}^{p+m-2}+C\d\E\int_{\mcO}\int_{\mcO}J(\z-\xi)|X_{r}^{\d}(\xi)-X_{r}^{\d}(\z)|^{m-2}d\z d\xi\\
 & \le-c\E\|X_{r}^{\d}\|_{p+m-2}^{p+m-2}+C\d(\E\|X_{r}^{\d}\|_{m}^{m}+1),
\end{align*}
for some $c>0.$ Now, using \eqref{eq:noise_bound} we obtain that
\begin{align*}
\sum_{k=1}^{\infty}\int_{\mcO}|X_{t}^{\d}|^{m-2}(Be_{k})^{2}d\z & \le\ve\|X_{t}^{\d}\|_{p+m-2}^{p+m-2}+C_{\ve},
\end{align*}
for all $\ve>0$ and some $C_{\ve}\ge0.$ Choosing $\ve,\d>0$ small
enough, we conclude that
\begin{align*}
\frac{1}{m}\E\|X_{t}^{\d}\|_{m}^{m}\le & \frac{1}{m}\E\|x_{0}\|_{m}^{m}-c\E\int_{0}^{t}\|X_{r}^{\d}\|_{p+m-2}^{p+m-2}dr+C\d\E\int_{0}^{t}\|X_{r}^{\d}\|_{m}^{m}dr+tC.
\end{align*}
Letting $\d\to0$ concludes the proof.
\end{proof}
We next analyze the deterministic situation, i.e.~$B=0$ in \eqref{eq:nonlocal_plp-1-1}.
Let $u$ be the unique SVI solution to 
\begin{align}
du_{t} & \in\left(\int_{\mcO}J(\cdot-\xi)|u_{t}(\xi)-u_{t}(\cdot)|^{p-2}(u_{t}(\xi)-u_{t}(\cdot))d\xi\right)dt\label{eq:nonlocal_plp-1-1-1}\\
u_{0} & =x_{0}\in H.\nonumber 
\end{align}

\begin{lem}
\label{lem:det_decay}Let $m_{0}=4-p\in(2,3]$, $x_{0}\in L_{av}^{m_{0}}(\mcO)\subseteq H$
and $u$ be the corresponding SVI solution to \eqref{eq:nonlocal_plp-1-1-1}.
Then there is a $C>0$ such that
\[
\|u_{t}\|_{2}^{2}\le\frac{C}{t}\|x_{0}\|_{m_{0}}^{m_{0}}.
\]
\end{lem}
\begin{proof}
From Lemma \ref{lem:general_lp_bound} we know that
\begin{align*}
\frac{1}{m_{0}}\|u_{t}\|_{m_{0}}^{m_{0}} & \le\frac{1}{m_{0}}\|x_{0}\|_{m_{0}}^{m_{0}}-c\int_{0}^{t}\|u_{r}\|_{p+m_{0}-2}^{p+m_{0}-2}dr\\
 & =\frac{1}{m_{0}}\|x_{0}\|_{m_{0}}^{m_{0}}-c\int_{0}^{t}\|u_{r}\|_{2}^{2}dr
\end{align*}
In particular, $t\mapsto\|u_{t}\|_{m_{0}}^{m_{0}}$ is non-increasing.
Using that also $t\mapsto\|u_{t}\|_{2}^{2}$ is non-increasing yields
\begin{align*}
\frac{1}{m_{0}}\|u_{t}\|_{m_{0}}^{m_{0}} & \le\frac{1}{m_{0}}\|x_{0}\|_{m_{0}}^{m_{0}}-ct\|u_{t}\|_{2}^{2}.
\end{align*}
Hence,
\[
\|u_{t}\|_{2}^{2}\le\frac{C}{t}\|x_{0}\|_{m_{0}}^{m_{0}}.
\]
\end{proof}
Next we prove concentration on bounded $L^{m_{0}}$ sets for sufficiently
regular initial conditions. 
\begin{lem}
\label{lem:concentration}Let $\ve>0$ and $x\in L_{\operatorname{av}}^{m_{1}}(\mcO)\subseteq H$
with $m_{1}=m_{0}+2-p\in(2,4]$, $m_{0}=4-p\in(2,3]$. Then there
is an $R=R(\ve)>0$ such that
\[
Q_{T}(x,B_{R}^{m_{0}}(0))\ge1-\ve
\]
for all $T\ge1$.
\end{lem}
\begin{proof}
By Lemma \ref{lem:general_lp_bound} we have
\begin{align*}
\frac{1}{t}\E\|X_{t}^{x}\|_{m_{1}}^{m_{1}}+c\E\frac{1}{t}\int_{0}^{t}\|X_{r}^{x}\|_{m_{0}}^{m_{0}}dr & \le\frac{1}{t}\E\|x\|_{m_{1}}^{m_{1}}+C.
\end{align*}
Thus, for $T\ge1$,
\begin{align*}
Q_{T}(x,B_{R}^{m_{0}}(0)) & =\frac{1}{T}\int_{0}^{T}P_{r}(x,B_{R}^{m_{0}}(0))dr\\
 & \ge\frac{1}{T}\int_{0}^{T}\left(1-\frac{\E\|X_{r}^{x}\|_{m_{0}}^{m_{0}}}{R}\right)dr\\
 & =1-\frac{1}{R}\frac{1}{T}\int_{0}^{T}\E\|X_{r}^{x}\|_{m_{0}}^{m_{0}}dr\\
 & \ge1-\frac{C}{RT}\E\|x\|_{m_{1}}^{m_{1}}-\frac{C}{R}.
\end{align*}
Choosing $R$ large enough yields the claim.
\end{proof}
\begin{lem}
\label{lem:compare_det}For each $T\ge0$, $\eta>0$ we have
\[
\inf_{x\in B}\P(\sup_{t\in[0,T]}\|X_{t}^{x}-u_{t}^{x}\|_{H}^{2}\le\eta)>0
\]
for all bounded sets $B\subseteq H$.
\end{lem}
\begin{proof}
We consider $Y_{t}^{\d}:=X_{t}^{\d}-W_{t}^{B}$ which satisfies
\begin{align*}
\frac{d}{dt}Y_{t}^{\d} & =A^{\d}(Y_{t}^{\d}+W_{t}^{B})dt\\
Y_{0}^{\d} & =x_{0}.
\end{align*}
Accordingly, let $u^{\d}$ be the unique solution to \eqref{eq:nonlocal_plp-1-1-1}
with $\phi(z)=|z|^{p-2}z$ replaced by $\phi^{\d}$. Then,
\begin{align*}
\frac{1}{2}\frac{d}{dt}\|Y_{t}^{\d}\|_{H}^{2} & =(Y_{t}^{\d},A^{\d}(Y_{t}^{\d}+W_{t}^{B}))_{H}\\
 & \le\|Y_{t}^{\d}\|_{H}^{2}+C\|W_{t}^{B}\|_{H}^{2}.
\end{align*}
Thus,
\[
\sup_{t\in[0,T]}\|Y_{t}^{\d}\|_{H}^{2}\lesssim1+\|x_{0}\|_{H}^{2}.
\]
Similarly,
\[
\sup_{t\in[0,T]}\|u_{t}^{\d}\|_{H}^{2}\lesssim1+\|x_{0}\|_{H}^{2}.
\]
Moreover, 
\begin{align*}
\frac{1}{2}\frac{d}{dt}\|Y_{t}^{\d}-u_{t}^{\d}\|_{H}^{2} & =(Y_{t}^{\d}-u_{t}^{\d},A^{\d}(Y_{t}^{\d}+W_{t}^{B})-A^{\d}(u_{t}^{\d}))_{H}\\
 & \le-(W_{t}^{B},A^{\d}(Y_{t}^{\d}+W_{t}^{B})-A^{\d}(u_{t}^{\d}))_{H}\\
 & \le\|W_{t}^{B}\|_{H}\|A^{\d}(Y_{t}^{\d}+W_{t}^{B})-A^{\d}(u_{t}^{\d})\|_{H}\\
 & \le C\|W_{t}^{B}\|_{H}(\|Y_{t}^{\d}\|_{H}+\|W_{t}^{B}\|_{H}+\|u_{t}^{\d}\|_{H})\\
 & \le C\|W_{t}^{B}\|_{H}(\|x_{0}\|_{H}+\|W_{t}^{B}\|_{H}+1).
\end{align*}
Since $W^{B}$ is a trace class Wiener process in $H$, for each $\eta\in(0,1],T>0$
we can find a subset $\O_{\eta}\subseteq\Omega$ of positive mass
such that $\sup_{t\in[0,T]}{\|W_{t}^{B}(\o)\|_{H}}<\eta$ for all
$\o\in\O_{\eta}$. For $\o\in\O_{\eta}$ we obtain
\begin{align*}
\frac{1}{2}\frac{d}{dt}\|Y_{t}^{\d}-u_{t}^{\d}\|_{H}^{2} & \le C\eta(\|x_{0}\|_{H}+1).
\end{align*}
Choosing $\eta>0$ small enough and letting $\d\to0$ yields the claim.
\end{proof}
\begin{lem}
\label{lem:mass_at_0}Let $\ve>0$, $x\in L_{\operatorname{av}}^{m_{1}}(\mcO)$
with $m_{1}$ as before and $\d>0$. Then
\[
\liminf_{T\to\infty}Q_{T}(x,B_{\d}(0))>0.
\]
\end{lem}
\begin{proof}
By Lemma \ref{lem:concentration} there is an $R>0$ such that
\[
Q_{T}(x,B_{R}^{m_{0}}(0))\ge\frac{1}{2}
\]
for all $T\ge1$. Moreover, by Lemma \ref{lem:det_decay} we have
\begin{align*}
\|u_{t}^{x}\|_{2}^{2} & \le\frac{C}{t}\|x\|_{m_{0}}^{m_{0}}\le\frac{C}{t}R
\end{align*}
for all $x\in B_{R}^{m_{0}}(0)$ and thus there is a $T_{0}=T_{0}(R,\d)$
such that
\begin{align*}
\|u_{t}^{x}\|_{2}^{2} & \le\frac{\d}{2},
\end{align*}
for all $t\ge T_{0}$. Using Lemma \ref{lem:compare_det} we observe
\[
P_{T_{0}}(x,B_{\d}(0))=P(\|X_{T_{0}}^{x}\|_{H}\le\d)\ge P(\|X_{T_{0}}^{x}-u_{T_{0}}^{x}\|_{H}\le\frac{\d}{2})\ge\g>0
\]
for some $\g=\g(\d,T_{0})>0$ and all $x\in B_{R}^{m_{0}}(0)$. Thus,
following an idea from \cite{ESR12}, we conclude that
\begin{align*}
\liminf_{T\to\infty}Q_{T}(x,B_{\d}(0)) & =\liminf_{T\to\infty}\frac{1}{T}\int_{0}^{T}P_{s}(x,B_{\d}(0))ds\\
 & =\liminf_{T\to\infty}\frac{1}{T}\int_{0}^{T}P_{s+T_{0}}(x,B_{\d}(0))ds\\
 & =\liminf_{T\to\infty}\frac{1}{T}\int_{0}^{T}\int_{H}P_{s}(x,dz)P_{T_{0}}(z,B_{\d}(0))ds\\
 & \ge\liminf_{T\to\infty}\frac{1}{T}\int_{0}^{T}\int_{B_{R}^{m_{0}}(0)}P_{s}(x,dz)P_{T_{0}}(z,B_{\d}(0))ds\\
 & \ge\g\liminf_{T\to\infty}Q_{T}(x,B_{R}^{m_{0}}(0))\\
 & \ge\frac{\g}{2}>0.
\end{align*}
\end{proof}
\begin{lem}
\label{lem:e-property}Let $X$, $Y$ be two SVI solutions to \eqref{eq:nonlocal_plp-1-1}
with initial conditions $x_{0},y_{0}\in L^{2}(\Omega,\mcF_{0};L_{\operatorname{av}}^{m}(\mcO))$
respectively. Then, for all $m\ge1$,
\[
\|X_{t}-Y_{t}\|_{L^{m}(\Ocal)}\le\|x_{0}-y_{0}\|_{L^{m}(\Ocal)}\quad\P\mathrm{{-a.s.}},\,\forall t\ge0.
\]
In particular, the semigroup $P_{t}$ satisfies the $e$-property
on $H$.
\end{lem}
\begin{proof}
We have that
\[
d(X_{t}^{\d}-Y_{t}^{\d})=(A^{\d}(X_{t}^{\d})-A^{\d}(Y_{t}^{\d}))dt
\]
and thus $t\mapsto(X_{t}^{\d}-Y_{t}^{\d})\in W^{1,2}([0,T];H)$. Let
$\iota^{\a}$ be the Moreau-Yosida approximation of $\frac{1}{m}|\cdot|^{m}$
and 
\[
\eta^{\a}(v):=\int_{\mcO}\iota^{\a}(v)d\z
\]
for $v\in H$. By \cite[Lemma IV.4.3]{S97} we obtain that
\begin{align*}
\frac{d}{dt}\eta^{\a}(X_{t}^{\d}-Y_{t}^{\d}) & =(\g_{t}^{\a},A^{\d}(X_{t}^{\d})-A^{\d}(Y_{t}^{\d}))_{H},\text{ for a.e. }t\in[0,T],
\end{align*}
where $\g_{t}^{\a}:=(\iota^{\a})'(X_{t}^{\d}-Y_{t}^{\d})\in L^{2}([0,T];H)$.
Using \cite[Lemma 6.6]{AVMRT10} we conclude that, $\P$-a.s.,
\begin{align*}
\frac{d}{dt}\eta^{\a}(X_{t}^{\d}-Y_{t}^{\d}) & \le0.
\end{align*}
Letting $\a\to0$, then $\d\to0$ concludes the proof.
\end{proof}

\begin{proof}
[Proof of Theorem \ref{thm:nonlocal_ergodicity}:] \textit{Step 1:}
Existence and uniqueness of invariant measures

The proof relies on an application of \cite[Theorem 1]{KPS10}. Let
$x\in H$ and $\d>0$. Then we may choose $y\in L_{\operatorname{av}}^{m_{1}}(\mcO)$,
with $m_{1}$ as in Lemma \ref{lem:concentration}, such that 
\[
\|x-y\|_{H}^{2}\le\frac{\d}{2}.
\]
By Lemma \ref{lem:e-property} we then have
\begin{equation}
\|X_{t}^{x}-X_{t}^{y}\|_{H}^{2}\le\frac{\d}{2}\quad\text{for all }t\ge0.\label{eq:diff}
\end{equation}
Lemma \ref{lem:mass_at_0} yields
\[
\liminf_{T\to\infty}Q_{T}(y,B_{\frac{\d}{2}}(0))>0.
\]
Due to \eqref{eq:diff} we conclude
\begin{equation}
\liminf_{T\to\infty}Q_{T}(x,B_{\d}(0))\ge\liminf_{T\to\infty}Q_{T}(y,B_{\frac{\d}{2}}(0))>0.\label{eq:concentration_at_0}
\end{equation}
An application of \cite[Theorem 1]{KPS10} implies that $P_{t}$ has
a unique invariant probability measure $\mu$.

\textit{Step 2: }We first note that for all $x\in H$ such that $\left\{ Q_{T}(x,\text{\textperiodcentered})\right\} _{T\ge T_{0}}$
is tight for some $T_{0}\ge0$, we have that
\[
Q_{T}(x,\text{\textperiodcentered})\rightharpoonup^{*}\mu\quad\text{for }T\to\infty
\]
by uniqueness of the invariant measure $\mu$. Hence,
\[
\mcT(\mu)=\left\{ x\in H:\,\left\{ Q_{T}(x,\text{\textperiodcentered})\right\} _{T\ge T_{0}}\,\text{is tight for some }T_{0}\ge0\right\} .
\]
By \cite[Proposition 1]{KPS10} we have that 
\[
\supp\mu\subseteq\mcT(\mu).
\]
Moreover, using invariance of $\mu$, Fatou's Lemma and \eqref{eq:concentration_at_0}
we note that
\begin{align*}
\mu(B_{\d}(0)) & =\liminf_{T\to\infty}Q_{T}\mu(B_{\d}(0))\\
 & =\liminf_{T\to\infty}\int_{H}Q_{T}(x,B_{\d}(0))d\mu(x)\\
 & \ge\int_{H}\liminf_{T\to\infty}Q_{T}(x,B_{\d}(0))d\mu(x)\\
 & >0
\end{align*}
for all $\d>0$. Hence, 
\[
0\in\operatorname{supp}\mu.
\]

\textit{Step 3:} An application of It\^{o}'s formula yields
\[
\E\|X_{t}^{\d}\|_{H}^{2}\le2\E\int_{0}^{t}(A^{\d}(X_{r}^{\d}),X_{r}^{\d})_{H}dr+t\|B\|_{L_{2}(H)}^{2}.
\]
By \cite[Lemma 6.5]{AVMRT10} we have
\[
2(A^{\d}(v),v)_{H}=-p\vp^{\d}(v)
\]
and thus
\[
\frac{c}{t}\E\int_{0}^{t}\vp^{\d}(X_{r}^{\d})dr\le\|B\|_{L_{2}(H)}^{2},
\]
for some $c>0$. Since, by \cite[Appendix A]{GT15}, 
\[
|\vp^{\d}(v)-\vp(v)|\le C\d(1+\|v\|_{H}^{2})\quad\forall v\in H
\]
we obtain that
\[
\frac{c}{t}\E\int_{0}^{t}\vp(X_{r}^{\d})dr\le\|B\|_{L_{2}(H)}^{2}+\frac{C\d}{t}\E\int_{0}^{t}(\|X_{r}^{\d}\|_{H}^{2}+1)dr.
\]
Letting $\d\to0$ yields
\[
\frac{c}{t}\E\int_{0}^{t}\vp(X_{r})dr\le\|B\|_{L_{2}(H)}^{2}.
\]
Since $0\in\mcT(\mu)$ this is easily seen to imply \eqref{eq:ipm_mass_bound}.
\end{proof}

\section{Ergodicity for stochastic local \texorpdfstring{$p$}{p}-Laplace
equations\label{sec:ergodicity_local}}

In this section we consider stochastic singular $p$-Laplace equations
with additive noise, that is,
\begin{align}
dX_{t} & \in\div\left(|\nabla X_{t}|^{p-2}\nabla X_{t}\right)\,dt+BdW_{t},\nonumber \\
|\nabla X_{t}|^{p-2}\nabla X_{t}\cdot\nu & \ni0\quad\mathrm{\text{on }}\partial\Ocal,\;t>0,\label{eq:singular_p_laplace-1}\\
X_{0} & =x_{0}\in L^{2}(\O,\mcF_{0};L_{\operatorname{av}}^{2}(\Ocal)),\nonumber 
\end{align}
with $p\in[1,2)$ on a bounded, smooth domain $\mcO\subseteq\R^{d}$
with convex boundary $\partial\mcO$. In the following we set $H:=L_{\operatorname{av}}^{2}(\Ocal)$
and $S:=H_{\operatorname{av}}^{1}(\mcO)$. Here, $W$ is a cylindrical
Wiener process on $H$ and $B\in L_{2}(H)$ symmetric with $B\in L_{2}(H,H_{\operatorname{av}}^{3})$.
Hence, 
\[
W_{t}^{B}:=BW_{t}
\]
is a trace-class Wiener process in $H_{\operatorname{av}}^{3}\subseteq H$.
As in Section \ref{sec:ergodicity_nonlocal}, we further assume that
there is an orthonormal basis $e_{k}$ of $H$ such that
\begin{equation}
\sum_{k=1}^{\infty}\|Be_{k}\|_{\infty}^{2}<\infty,\label{eq:noise_bound-1}
\end{equation}
cf.~\cite{BDPR09-4} where similar conditions on $B$ have been used
in the case $p=1$. We define, for $p\in(1,2)$,
\[
\vp(v):=\begin{cases}
\frac{1}{p}\int_{\Ocal}\abs{\nabla u}^{p}\,d\xi & \text{if }v\in W^{1,p}(\mcO)\\
+\infty & \text{if }v\in L^{p}(\mcO)\setminus W^{1,p}(\mcO)
\end{cases}
\]
and for $p=1$,
\[
\vp(v):=\begin{cases}
\|v\|_{TV} & \text{if }v\in BV(\mcO)\\
+\infty & \text{if }v\in L^{1}(\mcO)\setminus BV(\mcO).
\end{cases}
\]

Then \eqref{eq:singular_p_laplace-1} can be recast in its relaxed
form
\[
dX_{t}\in-\partial_{L^{2}}\vp(X_{t})dt+BdW_{t},
\]
where $\partial_{L^{2}}\vp$ denotes the $L^{2}$ subgradient of $\vp$
restricted to $L^{2}$. In \cite[Section 7.2.2]{GT11} the existence
and uniqueness of a (limit) solution $X=X^{x_{0}}$ to \eqref{eq:singular_p_laplace-1}
has been proven and 
\begin{equation}
\|X_{t}^{x}-X_{t}^{y}\|_{H}^{2}\le\|x-y\|_{H}^{2}\quad\forall t\ge0,\,\P\text{-a.s..}\label{eq:local_contraction}
\end{equation}
Following \cite[Appendix C]{GT11} it is easy to see that $X$ also
is an SVI solution to \eqref{eq:singular_p_laplace-1}, which by \cite[Section 3]{GT15}
is unique. From the construction of $X$ it is easy to see that the
average value is preserved, that is,
\begin{equation}
X_{t}\in L_{\operatorname{av}}^{2}(\Ocal)\quad\forall t\ge0,\,\P\text{-a.s.}\label{eq:X_av}
\end{equation}
if $x_{0}\in L^{2}(\O,\mcF_{0};L_{\operatorname{av}}^{2}(\Ocal))$.
Moreover, by \cite[Proposition 5.2]{GT11},
\[
P_{t}F(x):=\E F(X_{t}^{x})\quad\text{for }F\in\mcB_{b}(H)
\]
defines a Feller semigroup on $\mcB_{b}(H)$. By \eqref{eq:local_contraction},
$P_{t}$ satisfies the $e$-property on $H$. As a main result in
this section we obtain
\begin{thm}
\label{thm:local_ergodicity}There is a unique invariant measure $\mu$
for $P_{t}$, which satisfies 
\[
0\in\operatorname{supp}(\mu)\subseteq\mcT(\mu)
\]
and 
\[
\int_{H}\vp(x)d\mu(x)\lesssim\|B\|_{L_{^{2}}(H)}^{2}.
\]
\end{thm}
The proof of Theorem \ref{thm:local_ergodicity} proceeds along the
same principal ideas as Theorem \ref{thm:nonlocal_ergodicity}. However,
due to the local nature of \eqref{eq:singular_p_laplace-1} different
arguments have to be used in order to deduce the cascade of $L^{m}$
inequalities (cf.~Lemma \ref{lem:general_lp_bound-1} below). Once,
these inequalities have been shown for \eqref{eq:singular_p_laplace-1},
the proof can be concluded essentially as in Section \ref{sec:ergodicity_nonlocal}. 
\begin{lem}
\label{lem:general_lp_bound-1}Let $x_{0}\in L^{m}(\O,\mcF_{0};L_{\operatorname{av}}^{m}(\mcO))$,
$m\in[2,\infty)$ and let $X$ be the corresponding SVI solution to
\eqref{eq:nonlocal_plp-1-1}. Then there is a constant $c=c(p,m)>0$
such that
\begin{align}
\frac{1}{m}\E\|X_{t}\|_{m}^{m}+c\E\int_{0}^{t}\|X_{r}\|_{p+m-2}^{p+m-2}dr & \le\frac{1}{m}\E\|x_{0}\|_{m}^{m}+tC\quad\forall t\ge0.\label{eq:local_lp_bound}
\end{align}
\end{lem}
\begin{proof}
\textit{Step 1:} We start by proving that for $x_{0}\in L^{m}(\O,\mcF_{0};L_{\operatorname{av}}^{m}(\mcO))$
we have $\E\|X_{t}\|_{L^{m}}^{m}<\infty$ for all $t\ge0$. 

In the following let $\psi(\cdot)=\frac{1}{p}|\cdot|^{p}$, $\phi:=\partial\psi$,
$\psi^{\d}$ be the Moreau-Yosida approximation of $\psi$ and $\phi^{\d}:=\partial\psi^{\d}$.
Recall that the unique SVI solution $X$ to \eqref{eq:nonlocal_plp-1-1}
has been constructed in \cite[Theorem 4.1]{GT15} as a limit of approximating
solutions $X^{\ve,\d,n}$ to 
\begin{align}
dX_{t}^{\ve,\d,n} & =\ve\D X_{t}^{\ve,\d,n}\,dt+\div\phi^{\d}\left(\nabla X_{t}^{\ve,\d,n}\right)\,dt+BdW_{t},\label{eq:full_approx_SVI_constr-1}\\
X_{0}^{\ve,\d,n} & =x_{0}^{n},\nonumber 
\end{align}
with zero Neumann boundary conditions, where $x_{0}^{n}\to x_{0}$
in $L^{2}(\O;H)$ with $x_{0}^{n}\in L^{2}(\O,\mcF_{0};H_{\operatorname{av}}^{1})$
and $\ve>0$. From \cite[Equation (3.6)]{GT15} we recall the bound
\begin{equation}
\E\sup_{t\in[0,T]}\|X_{t}^{\ve,\d,n}\|_{H^{1}}^{2}+2\ve\E\int_{0}^{T}\|\D X_{r}^{\ve,\d,n}\|_{2}^{2}dr\le C(\E\|x_{0}^{n}\|_{H^{1}}^{2}+1),\label{eq:strong_soln_visc-1}
\end{equation}
with a constant $C>0$ independent of $\ve$, $\d$ and $n$. In \cite[proof of Theorem 3.1]{GT15}
the following subsequent convergence has been shown in $L^{2}(\O;C([0,T];H))$
\begin{align}
X^{\ve,\d,n} & \to X^{\ve,n}\quad\text{for }\d\to0\nonumber \\
X^{\ve,n} & \to X^{n}\quad\text{for }\ve\to0\label{eq:convergence}\\
X^{n} & \to X\quad\text{for }n\to\infty.\nonumber 
\end{align}
Let $\iota^{\a},\t^{\b}$ and $\eta^{\a,\b}$ be as in the proof of
Lemma \ref{lem:general_lp_bound}. Then, since $X^{\ve,\d,n}$ is
a strong solution to \eqref{eq:full_approx_SVI_constr-1} and by It\^{o}'s
formula
\begin{align}
\E\eta^{\a,\b}( & X_{t}^{\ve,\d,n})=\E\eta^{\a,\b}(x_{0})\nonumber \\
 & +\E\int_{0}^{t}\int_{\mcO}(\iota^{\a})'(\t^{\b}\ast X_{r}^{\ve,\d,n})\left(\t^{\b}\ast(\ve\D X_{r}^{\ve,\d,n}+\div\phi^{\d}\left(\nabla X_{r}^{\ve,\d,n}\right))\right)\,d\xi dr\label{eq:approx_ito-1}\\
 & +\sum_{k=1}^{\infty}\E\int_{0}^{t}\int_{\mcO}(\iota^{\a})''(\t^{\b}\ast X_{r}^{\ve,\d,n})(\t^{\b}\ast Be_{k})^{2}d\xi dr\nonumber 
\end{align}
Using \eqref{eq:noise_bound-1}, \eqref{eq:strong_soln_visc-1} and
dominated convergence, taking $\b\to0$ yields
\begin{align}
\E\int_{\mcO}\iota^{\a}(X_{t}^{\ve,\d,n})d\xi & =\E\int_{\mcO}\iota^{\a}(x_{0})d\xi\nonumber \\
 & +\E\int_{0}^{t}\int_{\mcO}(\iota^{\a})'(X_{r}^{\ve,\d,n})(\ve\D X_{r}^{\ve,\d,n}+\div\phi^{\d}\left(\nabla X_{r}^{\ve,\d,n}\right))\,d\xi dr\label{eq:approx_ito}\\
 & +\sum_{k=1}^{\infty}\E\int_{0}^{t}\int_{\mcO}(\iota^{\a})''(X_{r}^{\ve,\d,n})(Be_{k})^{2}d\xi dr\nonumber 
\end{align}
Since $\int_{\mcO}(\iota^{\a})'(X_{r}^{\ve,\d,n})(\ve\D X_{r}^{\ve,\d,n}+\div\phi^{\d}\left(\nabla X_{r}^{\ve,\d,n}\right))\,d\xi\le0$,
using \eqref{eq:iota_der}, this implies 
\begin{align*}
 & \E\int_{\mcO}\iota^{\a}(X_{t}^{\ve,\d,n})d\xi\le\E\int_{\mcO}\iota^{\a}(x_{0})d\xi+C\E\int_{0}^{t}\int_{\mcO}1+\iota^{\a}(X_{r}^{\ve,\d,n})dxdr.
\end{align*}
Hence, by Gronwall's Lemma 
\begin{align*}
\E\int_{\mcO}\iota^{\a}(X_{t}^{\ve,\d,n})d\xi & \lesssim\E\int_{\mcO}\iota^{\a}(x_{0})d\xi+1\lesssim\frac{1}{m}\E\|x_{0}\|_{m}^{m}+1.
\end{align*}
Taking the limit $\a\to0$ yields, by Fatou's Lemma and using dominated
convergence
\begin{align}
\frac{1}{m}\E\|X_{t}^{\ve,\d,n}\|_{m}^{m} & \lesssim\frac{1}{m}\E\|x_{0}\|_{m}^{m}+1.\label{eq:lm-first-bound}
\end{align}
Taking the limits $\d\to0$, $\ve\to0$, $n\to\infty$ subsequently
as in the proof of \cite[Theorem 3.1]{GT15} finishes the proof of
this step by Fatou's Lemma.

\textit{Step 2: }As in the proof of Lemma \ref{lem:general_lp_bound}
it is enough to prove \eqref{eq:local_lp_bound} for $x_{0}\in L^{\infty}(\O,\mcF_{0};L_{\operatorname{av}}^{\infty}(\mcO))$.
Hence, assume $x_{0}\in L^{\infty}(\O,\mcF_{0};L_{\operatorname{av}}^{\infty}(\mcO))$
from now on. By step one we have $\E\|X_{t}\|_{m}^{m}<\infty$ for
all $t\ge0$, $m\in\N$.

Taking the limit $\a\to0$ in \eqref{eq:approx_ito} yields, using
dominated convergence and \eqref{eq:iota_der}, \eqref{eq:lm-first-bound},
\begin{align*}
 & \frac{1}{m}\E\int_{\mcO}\|X_{t}^{\ve,\d,n}\|_{m}^{m}d\xi\\
 & \le\frac{1}{m}\E\int_{\mcO}\|x_{0}\|_{m}^{m}+\E\int_{0}^{t}\int_{\mcO}(X_{r}^{\ve,\d,n})^{[m-1]}(\ve\D X_{r}^{\ve,\d,n}+\div\phi^{\d}\left(\nabla X_{r}^{\ve,\d,n}\right))\,d\xi dr\\
 & +C\E\int_{0}^{t}\|X_{r}^{\ve,\d,n}\|_{m}^{m}dr+Ct.
\end{align*}
We observe that, for $v\in H^{2}$ with $\nabla v\cdot\nu=0$ on $\partial\mcO$,
\begin{align*}
\int_{\mcO}v^{[m-1]}\div\phi^{\d}\left(\nabla v\right)d\z & =-(m-1)\int_{\mcO}|v|^{m-2}\nabla v\cdot\phi^{\d}(\nabla v)d\z.
\end{align*}
Further, note that (since $\psi^{\d}(0)=0$)
\begin{align*}
a\cdot\phi^{\d}(a) & =a\cdot\partial\psi^{\d}(a)\\
 & \ge\psi^{\d}(a)\quad\forall a\in\R^{d}
\end{align*}
and (cf. \cite[Appendix A]{GT15})
\begin{align*}
|\psi^{\d}(a)-\psi(a)| & \lesssim\d(1+\psi(a))\quad\forall a\in\R^{d}.
\end{align*}
This implies 
\[
a\cdot\phi^{\d}(a)\ge c\psi(a)-C\d.
\]
and thus, for $\d>0$ small enough,
\begin{align*}
 & \int_{\mcO}(X_{r}^{\ve,\d,n})^{[m-1]}\div\phi^{\d}\left(\nabla X_{r}^{\ve,\d,n}\right)d\z\\
 & =-(m-1)\int_{\mcO}|X_{r}^{\ve,\d,n}|^{m-2}\nabla X_{r}^{\ve,\d,n}\cdot\phi^{\d}\left(\nabla X_{r}^{\ve,\d,n}\right)d\z\\
 & \le-c(m-1)\int_{\mcO}|X_{r}^{\ve,\d,n}|^{m-2}\psi(\nabla X_{r}^{\ve,\d,n})d\z+C(m-1)\d\int_{\mcO}|X_{r}^{\ve,\d,n}|^{m-2}d\z.
\end{align*}
For $u\in W_{\operatorname{av}}^{1,\infty}$ we observe that
\begin{align*}
\int_{\mcO}|u|^{m-2}\psi(\nabla u)d\z & =\frac{1}{p}\int_{\mcO}|u|^{m-2}|\nabla u|^{p}d\z\\
 & =\frac{1}{p}\int_{\mcO}||u|^{\frac{m-2}{p}}\nabla u|^{p}d\z\\
 & =c\int_{\mcO}|\nabla u{}^{[\frac{m-2+p}{p}]}|^{p}d\z,
\end{align*}
for some generic constant $c=c(m,p)>0$. By Poincar\'{e}'s inequality
we obtain that
\begin{align*}
\int_{\mcO}|u|^{m-2}\psi(\nabla u)d\z & \ge c\int_{\mcO}|u|^{p+m-2}d\z.
\end{align*}
By smooth approximation, this inequality remains true for all $u\in H_{\operatorname{av}}^{1}$
with $u\in\bigcap_{m\ge1}L^{m}$. Hence, using step one and \eqref{eq:X_av}
we conclude 
\begin{align*}
\E\int_{\mcO}|X_{r}^{\ve,\d,n}|^{m-2}\psi(\nabla X_{r}^{\ve,\d,n})d\z & \ge c\E\int_{\mcO}|X_{r}^{\ve,\d,n}|^{p+m-2}d\z.
\end{align*}
Using this above yields that
\begin{align*}
 & \frac{1}{m}\E\|X_{t}^{\ve,\d,n}\|_{m}^{m}+c\E\int_{0}^{t}\|X_{r}^{\ve,\d,n}\|_{p+m-2}^{p+m-2}dr\\
 & \le\frac{1}{m}\E\|x_{0}\|_{m}^{m}+C\E\int_{0}^{t}\|X_{r}^{\ve,\d,n}\|_{m}^{m}dr+Ct+C(m-1)\d\E\int_{\mcO}|X_{r}^{\ve,\d,n}|^{m-2}d\z.
\end{align*}
By step one 
\[
\d\E\int_{\mcO}|X_{r}^{\ve,\d,n}|^{m-2}d\z\to0
\]
for $\d\to0$. In conclusion, by Fatou's Lemma and \eqref{eq:convergence},
\begin{align*}
 & \frac{1}{m}\E\|X_{t}^{\ve,n}\|_{m}^{m}+c\int_{0}^{t}\E\|X_{r}^{\ve,n}\|_{p+m-2}^{p+m-2}dr\le\frac{1}{m}\E\|x_{0}\|_{m}^{m}+C\E\int_{0}^{t}\|X_{r}^{\ve,n}\|_{m}^{m}dr+Ct.
\end{align*}
An application of Gronwall's Lemma and then letting $\ve\to0$, $n\to\infty$
concludes the proof.
\end{proof}
The proof may now be concluded as in Section \ref{sec:ergodicity_nonlocal}.
For the readers convenience we give some details. Let $u$ be the
unique solution to 
\begin{align}
du_{t} & \in\div\left(|\nabla u_{t}|^{p-2}\nabla u_{t}\right)\,dt,\nonumber \\
|\nabla u_{t}|^{p-2}\nabla u_{t}\cdot\nu & \ni0\quad\mathrm{{on}}\;\partial\Ocal,\;t>0,\label{eq:singular_p_laplace-1-1}\\
u_{0} & =x_{0}\in H,\nonumber 
\end{align}

\begin{lem}
\label{lem:det_decay-1}Let $x_{0}\in H$ and $u$ be the corresponding
solution to \eqref{eq:singular_p_laplace-1-1}. Then there is a $C>0$
such that
\[
\|u_{t}\|_{2}^{2}\le\frac{C}{t}\|x_{0}\|_{m_{0}}^{m_{0}},
\]
where $m_{0}=4-p\in(2,3]$. 
\end{lem}
\begin{proof}
Using Lemma \ref{lem:general_lp_bound-1}, the proof is analogous
to Lemma \ref{lem:det_decay}.
\end{proof}
\begin{lem}
\label{lem:concentration-1}Let $\ve>0$ and $x\in L^{m_{1}}(\O,\mcF_{0};L^{m_{1}}(\mcO))$
with $m_{1}=m_{0}+2-p\in(2,4]$, $m_{0}=4-p\in(2,3]$. Then there
is an $R=R(\ve)>0$ such that
\[
Q_{T}(x,B_{R}^{m_{0}}(0))\ge1-\ve
\]
for all $T\ge1$.
\end{lem}
\begin{proof}
Using Lemma \ref{lem:general_lp_bound-1}, the proof is analogous
to Lemma \ref{lem:concentration}.
\end{proof}
\begin{lem}
\label{lem:compare_det-1}For each $T\ge0$, $\d>0$ we have
\[
\inf_{x\in B}\P\left(\sup_{t\in[0,T]}\|X_{t}^{x}-u_{t}^{x}\|_{H}^{2}\le\d\right)>0
\]
for all bounded sets $B\subseteq H$.
\end{lem}
\begin{proof}
Follows from \cite[Lemma 6.6]{GT15}.
\end{proof}
\begin{lem}
\label{lem:mass_at_0-1}Let $\ve>0$, $x\in L^{m_{1}}(\mcO)$ with
$m_{1}$ as before and $\d>0$. Then
\[
\liminf_{T\to\infty}Q_{T}(x,B_{\d}(0))>0.
\]
\end{lem}
\begin{proof}
Same as Lemma \ref{lem:mass_at_0}. 
\end{proof}

\begin{proof}[Proof of Theorem \ref{thm:local_ergodicity}]
 Same as Theorem \ref{thm:nonlocal_ergodicity}.
\end{proof}

\section{Convergence of solutions: Non-local to local\label{sec:to_local_sol}}

In this section we investigate the convergence of the solutions to
the stochastic nonlocal $p$-Laplace equation to solutions of the
stochastic (local) $p$-Laplace equation, under appropriate rescaling
of the kernel $J$. The convergence of the associated unique invariant
measures will be considered in Section \ref{secc:convergence_measures}
below.

In the following let $\Ocal\subset\R^{d}$ be a bounded, smooth domain
with convex boundary $\partial\mcO$ and let $J:\R^{d}\to\R$ be a
nonnegative continuous radial function with compact support, $J(0)>0$,
$\int_{\R^{d}}J(z)\,dz=1$ and $J(x)\ge J(y)$ for all $|x|\le|y|$.
Further, let $W$ be a cylindrical Wiener process on $H$ and $B\in L_{2}(H)$
symmetric with $B\in L_{2}(H,H_{\operatorname{av}}^{3})$. As above,
let $H:=L_{\operatorname{av}}^{2}(\mcO)$ and $S=H_{\operatorname{av}}^{1}(\mcO)$.

For $p\in(1,2)$, $\eps>0$, we consider the rescaled stochastic nonlocal
$p$-Laplace equations of the type 
\begin{align}
dX_{t}^{\ve} & =\left(\int_{\mcO}J^{\ve}(\cdot-\xi)|X_{t}^{\ve}(\xi)-X_{t}^{\ve}(\cdot)|^{p-2}(X_{t}^{\ve}(\xi)-X_{t}^{\ve}(\cdot))d\xi\right)dt+BdW_{t}\label{eq:nonlocal_plp-1-1-2}\\
X_{0}^{\ve} & =x_{0}\in L^{2}(\O,\mcF_{0};H)\nonumber 
\end{align}
with 
\[
J^{\ve}(\xi):=\frac{C_{J,p}}{\eps^{p+d}}J\left(\frac{\xi}{\eps}\right),\quad\xi\in\R^{d}
\]
and corresponding energy 
\[
\vp^{\ve}(u):=\frac{1}{2p}\int_{\Ocal}\int_{\Ocal}J^{\ve}\left(\xi-\zeta\right)\lrabs{u(\zeta)-u(\xi)}^{p}\,d\zeta d\xi,
\]
for $u\in L^{p}(\mcO)$, where 
\[
C_{J,p}^{-1}:=\frac{1}{2}\int_{\R^{d}}J(z)|z_{d}|^{p}\,dz.
\]

Furthermore, we set 
\[
\vp(u):=\begin{cases}
\frac{1}{p}\int_{\Ocal}\abs{\nabla u}^{p}\,d\xi, & \quad\text{if}\;\;u\in W^{1,p}(\Ocal),\\
+\infty, & \quad\text{if}\;\;u\in L^{p}(\Ocal)\setminus W^{1,p}(\Ocal).
\end{cases}
\]
By \cite[Theorem 4.1]{GT15}, for each $\ve>0$, there is a unique
SVI solution $X^{\ve}$ to the stochastic nonlocal $p$-Laplace equation
\begin{align}
dX_{t}^{\ve} & =-\partial_{L^{2}}\vp^{\ve}(X_{t}^{\ve})\,dt+BdW_{t},\label{eq:nonlocal_plp}\\
X_{0}^{\ve} & =x_{0}\nonumber 
\end{align}
and, by \cite[Theorem 3.1]{GT15}, there is a unique SVI solution
to the stochastic (local) $p$-Laplace equation
\begin{align}
dX_{t} & =-\partial_{L^{2}}\varphi(X_{t})\,dt+BdW_{t},\label{eq:local_plp}\\
X_{0} & =x_{0},\nonumber 
\end{align}
where $\partial_{L^{2}}\varphi$ denotes the $L^{2}$ subgradient
of $\vp$ restricted to $L^{2}$. In \cite[Section 5]{GT15} weak
convergence of $X^{\ve}\rightharpoonup X$ in $L^{2}([0,T]\times\O;H)$
has been shown. The aim of this section is to strengthen this to pointwise
in time weak convergence, that is, 
\begin{equation}
X_{t}^{\ve}\rightharpoonup X_{t}\quad\text{weakly in }H\text{ for all }t\ge0,\P\text{-a.s.}.\label{eq:weak_conv}
\end{equation}
This will be crucial in order to obtain the convergence of the associated
semigroups $P_{t}^{\ve}F(x)=\E F(X_{t}^{\ve,x})$ to $P_{t}F(x)=\E F(X_{t}^{x})$
for cylindrical functions $F\in\mcF C_{b}^{1}(H)$. 

The strategy to prove the pointwise weak convergence \eqref{eq:weak_conv}
is based on considering the transformed equations (cf.~Remark \ref{rem:transformation}
below)
\begin{equation}
\frac{{d}}{dt}Y_{t}=-\partial_{L^{2}}\vp(Y_{t}+W_{t}^{B}(\omega))=A(Y_{t}+W_{t}^{B}(\omega))\label{eq:det_eq_transformed-2}
\end{equation}
and
\begin{equation}
\frac{{d}}{dt}Y_{t}^{\ve}=-\partial_{L^{2}}\vp^{\ve}(Y_{t}^{\ve}+W_{t}^{B}(\omega))=A^{\ve}(Y_{t}^{\ve}+W_{t}^{B}(\omega))\label{eq:det_eq_transformed-1-1}
\end{equation}
and to prove the weak convergence $Y_{t}^{\ve}\rightharpoonup Y_{t}$
in $H$ for all $t\ge0$ and a.a.~$\o\in\O$. The advantage of considering
the transformed, random PDE \eqref{eq:det_eq_transformed-2} and \eqref{eq:det_eq_transformed-1-1}
is that $Y^{\ve},Y$ enjoy better time regularity properties than
$X^{\ve},X$ which may be used to deduce stronger convergence results. 
\begin{thm}
\label{thm:weakly-pointwise}Let $x_{0}\in L^{2}(\O,\mcF_{0};H)$
and $X^{\ve}$, $X$ be the unique solutions to \eqref{eq:nonlocal_plp},
\eqref{eq:local_plp} respectively. Then, for each sequence $\ve_{n}\to0$,
\[
X_{t}^{\ve_{n}}\rightharpoonup X_{t}\quad\text{for }n\to\infty
\]
weakly in $H$ for all \textup{$t\in[0,T],\P$}-a.s.\textup{.} In
particular,
\[
P_{t}^{\ve}F(x)\to P_{t}F(x)\quad\text{for }\ve\to0
\]
for all $F\in\mcF C_{b}^{1}(H)$, $t\in[0,T]$, $x\in H$. 
\end{thm}
Motivated by \cite{GT15}, the general strategy of the proof of Theorem
\ref{thm:weakly-pointwise} is based on the SVI framework. Hence,
we first briefly sketch well-posedness of SVI solutions to \eqref{eq:det_eq_transformed-2}
and then proceed with the proof of Theorem \ref{thm:weakly-pointwise}. 
\begin{rem}
In this section, we restrict to the case $p\in(1,2)$ for simplicity
only. The interested reader will notice that the same arguments can
be applied in the case $p=1$ with only minor changes. The only difference
is the treatment of the nonlocal, transformed random PDE \eqref{eq:det_eq_transformed-1-1}.
In the case $p=1$ well-posedness of SVI solutions to \eqref{eq:det_eq_transformed-1-1}
has to be shown as a first step. This can be done following the same
arguments as in \cite[Section 3]{GT15}. 
\end{rem}

\subsection{SVI approach to the transformed equation \texorpdfstring{\eqref{eq:det_eq_transformed-2}}{}}

Let $H:=L_{\operatorname{av}}^{2}(\mcO)$, $S=H_{\operatorname{av}}^{1}(\mcO)$.
Without loss of generality we may assume $W^{B}(\o)\in C([0,T];H_{\operatorname{av}}^{3})$
for all $\o\in\O$. In analogy to \cite[Definition 2.1]{GT15} we
define
\begin{defn}
\label{def:SPDE_SVI-1}Let $x_{0}\in H$, $T>0.$ A map $Y\in L^{2}([0,T];H)$
is said to be an\emph{ SVI solution} to \eqref{eq:det_eq_transformed-2}
if 

\begin{enumerate}
\item {[}Regularity{]}
\begin{equation}
\esssup_{t\in[0,T]}\|Y_{t}\|_{H}^{2}+\int_{0}^{T}\vp(Y_{r}+W_{r}^{B})dr\le\|x_{0}\|_{H}^{2}+C,\label{eq:pathw_SVI_energy}
\end{equation}
for some constants $C>0.$
\item {[}Variational inequality{]} for each map $Z\in W^{1,2}([0,T];S)$
we have
\begin{align}
 & \|Y_{t}-Z_{t}\|_{H}^{2}+2\int_{0}^{t}\vp(Y_{r}+W_{r}^{B})\,dr\label{eq:SVI-3}\\
 & \le\E\|x_{0}-Z_{0}\|_{H}^{2}+2\int_{0}^{t}\vp(Z_{r}+W_{r}^{B})\,dr-2\int_{0}^{t}(G_{r},Y_{r}-Z_{r})_{H}\,dr,\nonumber 
\end{align}
for a.e.~$t\in[0,T]$, where $G:=\frac{d}{dt}Z$.
\end{enumerate}
If, in addition, $Y\in C([0,T];H)$ then $Y$ is said to be a (time-)
continuous SVI solution to \eqref{eq:det_eq_transformed-2}.
\end{defn}
\begin{prop}
\label{prop:det:SVI}For each $x_{0}\in H$, $\o\in\O$ there is a
unique time-continuous SVI solution $Y=Y(\o)$ to \eqref{eq:det_eq_transformed-2}.
The process $(t,\o)\mapsto Y_{t}(\o)$ is $\mcF_{t}$-progressively
measurable and the constant $C=C(\o)$ in \eqref{eq:pathw_SVI_energy}
satisfies $C\in L^{2}(\O)$.
\end{prop}
\begin{proof}
The proof follows the same line of arguments as the proof of \cite[Theorem 3.1]{GT15}.
Hence, in the following we shall restrict to giving some details on
required modifications of the proof. For notational convenience we
set
\begin{align*}
\psi(z):= & \frac{1}{p}|z|^{p},\\
\phi(z):= & \partial\psi(z)
\end{align*}
and let $\psi^{\d}$ be the Moreau-Yosida approximation of $\psi$,
$\phi^{\d}:=\partial\psi^{\d}$. In analogy to \eqref{eq:full_approx_SVI_constr-1}
we consider the three step approximation
\begin{align}
\frac{{d}}{dt}Y_{t}^{\ve,\d,n} & =\ve\D Y_{t}^{\ve,\d,n}+\div\phi^{\d}\left(\nabla(Y_{t}^{\ve,\d,n}+W_{t}^{B}(\omega))\right)\label{eq:det_eq_transformed-3}\\
Y_{0}^{\ve,\d,n} & =x_{n}\in H_{\operatorname{av}}^{1}.\nonumber 
\end{align}
Existence and uniqueness of a variational solution $Y^{\ve,\d,n}$
to \eqref{eq:det_eq_transformed-3} follows easily from \cite[Theorem 4.2.4]{PR07}.
Progressive measurability of $(t,\o)\mapsto Y_{t}^{\ve,\d,n}(\o)$
follows as in \cite[proof of Theorem 1.1]{GLR11}.

\emph{Step 1:} The first step consists in proving the existence of
a strong solution to \eqref{eq:det_eq_transformed-3} and corresponding
(uniform) energy bounds as in \eqref{eq:strong_soln_visc-1}. We restrict
to an informal derivation of these estimates, the rigorous justification
proceeds as in \cite[Theorem 3.1]{GT15}. We set $\|v\|_{\dot{H}^{1}}^{2}:=\|\nabla v\|_{2}^{2}$
for $v\in H^{1}$. Informally, we compute
\begin{align*}
 & \frac{{d}}{dt}\|Y_{t}^{\ve,\d,n}\|_{\dot{H}^{1}}^{2}\\
 & =-2(\ve\D Y_{t}^{\ve,\d,n}+\div\phi^{\d}\left(\nabla(Y_{t}^{\ve,\d,n}+W_{t}^{B}(\omega))\right),\D Y_{t}^{\ve,\d,n})_{H}\\
 & =-2\ve\|\D Y_{t}^{\ve,\d,n}\|_{H}^{2}-2(\div\phi^{\d}\left(\nabla(Y_{t}^{\ve,\d,n}+W_{t}^{B}(\omega))\right),\D(Y_{t}^{\ve,\d,n}+W_{t}^{B}(\omega)))_{H}\\
 & +2(\div\phi^{\d}\left(\nabla(Y_{t}^{\ve,\d,n}+W_{t}^{B}(\omega))\right),\D W_{t}^{B}(\omega))_{H}.
\end{align*}
For $v\in H^{2}$ with $\nabla v\cdot\nu=0$ on $\partial\mcO$,
arguing as in \cite[Example 7.11]{GT11}, we obtain that 
\begin{align}
-(v,\div\phi^{\d}(\nabla v))_{\dot{H}^{1}} & =-(-\D v,\div\phi^{\d}(\nabla v))_{2}\nonumber \\
 & =-\lim_{n\to\infty}(T_{n}v,\div\phi^{\d}(\nabla v))_{2}\nonumber \\
 & =-\lim_{n\to\infty}(nu-nJ_{n}u,\div\phi^{\d}(\nabla v))_{2}\label{eq:liu-ref}\\
 & \le\lim_{n\to\infty}n\left(\int_{\mcO}\psi^{\d}(\nabla J_{n}u)d\xi-\int_{\mcO}\psi^{\d}(\nabla u)d\xi\right)\nonumber \\
 & \le0,\nonumber 
\end{align}
where $T_{n}$ is the Yosida-approximation and $J_{n}$ the resolvent
of the Neumann Laplacian $-\D$ on $L^{2}$. Here, the convexity of
$\Ocal$ is needed, see \cite{GT11} for details. We next observe
that
\begin{align*}
 & 2(\div\phi^{\d}\left(\nabla(Y_{t}^{\ve,\d,n}+W_{t}^{B}(\omega))\right),\D W_{t}^{B}(\omega))_{H}\\
 & =2\int_{\mcO}\div\phi^{\d}\left(\nabla(Y_{t}^{\ve,\d,n}+W_{t}^{B}(\omega))\right)\D W_{t}^{B}(\omega)d\xi\\
 & \le C\int_{\mcO}(1+|\nabla(Y_{t}^{\ve,\d,n}+W_{t}^{B}(\omega))|)|\nabla\D W_{t}^{B}(\omega)|d\xi\\
 & \le C\left(\|\nabla Y_{t}^{\ve,\d,n}\|_{H}^{2}+1+\|W_{t}^{B}(\omega)\|_{H^{3}}^{2}\right).
\end{align*}
Hence,
\begin{align*}
\frac{{d}}{dt}\|Y_{t}^{\ve,\d,n}\|_{\dot{H}^{1}}^{2} & \le-2\ve\|\D Y_{t}^{\ve,\d,n}\|_{H}^{2}+C(\|Y_{t}^{\ve,\d,n}\|_{\dot{H}^{1}}^{2}+1+\|W_{t}^{B}(\omega)\|_{H^{3}}).
\end{align*}
By Gronwall's lemma this implies
\begin{align}
 & \sup_{t\in[0,T]}\|Y_{t}^{\ve,\d,n}\|_{\dot{H}^{1}}^{2}+2\ve\int_{0}^{T}\|\D Y_{r}^{\ve,\d,n}\|_{H}^{2}dr\nonumber \\
 & \le C\left(\|x_{0}\|_{\dot{H}^{1}}^{2}+\int_{0}^{T}\|W_{r}^{B}(\omega)\|_{H^{3}}dr+1\right)\label{eq:Y-energy-bound}\\
 & \le C(\|x_{0}\|_{\dot{H}^{1}}^{2}+1),\nonumber 
\end{align}
which finishes the proof of the required energy bound. 

\emph{Step 2: }\textit{\emph{We next derive the variational inequality
\eqref{eq:SVI-3} and regularity estimate \eqref{eq:pathw_SVI_energy}
for $Y^{\ve,\d,n}$.}} By the chain-rule and since $\phi^{\d}=\partial\psi^{\d}$
we have that
\begin{align*}
 & \norm{Y_{t}^{\ve,\d,n}-Z_{t}}_{H}^{2}\\
= & \norm{x_{0}^{n}-Z_{0}}_{H}^{2}+2\int_{0}^{t}(\div\phi^{\d}\left(\nabla(Y_{r}^{\ve,\d,n}+W_{r}^{B}(\omega))\right)+\varepsilon\Delta Y_{r}^{\ve,\d,n}-G_{r},Y_{r}^{\ve,\d,n}-Z_{r})_{H}dr\\
= & \norm{x_{0}^{n}-Z_{0}}_{H}^{2}+2\int_{0}^{t}\int_{\mcO}\phi^{\d}\left(\nabla(Y_{r}^{\ve,\d,n}+W_{r}^{B}(\omega))\right)\cdot(\nabla Z_{r}-\nabla Y_{r}^{\ve,\d,n})d\xi dr\\
 & -2\varepsilon\int_{0}^{t}(\Delta Y_{r}^{\ve,\d,n},Z_{r}-Y_{r}^{\ve,\d,n})_{H}dr-2\int_{0}^{t}(G_{r},Y_{r}^{\ve,\d,n}-Z_{r})_{H}dr\\
\le & \norm{x_{0}^{n}-Z_{0}}_{H}^{2}-2\int_{0}^{t}\left(\int_{\mcO}\psi^{\d}\left(\nabla(Y_{r}^{\ve,\d,n}+W_{r}^{B}(\omega))\right)d\xi+\frac{\varepsilon}{2}\norm{Y_{r}^{\varepsilon,\delta,n}}_{\dot{H}^{1}}^{2}\right)dr\\
 & +\varepsilon\int_{0}^{t}\norm{Z_{r}}_{\dot{H}^{1}}^{2}dr+2\int_{0}^{t}\int_{\mcO}\psi^{\d}\left(\nabla(Z_{r}+W_{r}^{B}(\omega))\right)d\xi dr-2\int_{0}^{t}(G_{r},Y_{r}^{\ve,\d,n}-Z_{r})_{H}dr,
\end{align*}
for all $Z\in W^{1,2}([0,T];H)$ and $G:=\frac{d}{dt}Z$. In particular,
choosing $Z\equiv0$ we obtain that 
\begin{align*}
 & \norm{Y_{t}^{\ve,\d,n}}_{H}^{2}+2\int_{0}^{t}\int_{\mcO}\psi^{\d}\left(\nabla(Y_{r}^{\ve,\d,n}+W_{r}^{B}(\omega))\right)d\xi\,dr\\
 & \le\norm{x_{0}}_{H}^{2}+2\int_{0}^{t}\int_{\mcO}\psi^{\d}\left(\nabla W_{r}^{B}(\omega)\right)\,dr\\
 & \le\norm{x_{0}}_{H}^{2}+C(\o),
\end{align*}
with $C\in L^{2}(\O)$.

\emph{Step 3:} The next step is to take the limit $\d\to0$, i.e.~we
estimate
\begin{align*}
 & \frac{{d}}{dt}\|Y_{t}^{\ve,\d_{1},n}-Y_{t}^{\ve,\d_{2},n}\|_{H}^{2}\\
 & =2\ve(\D Y_{t}^{\ve,\d_{1},n}-\D Y_{t}^{\ve,\d_{2},n},Y_{t}^{\ve,\d_{1},n}-Y_{t}^{\ve,\d_{2},n})_{H}\\
 & +2(\div\phi^{\d_{1}}\left(\nabla(Y_{t}^{\ve,\d_{1},n}+W_{t}^{B}(\omega))\right)-\div\phi^{\d_{2}}\left(\nabla(Y_{t}^{\ve,\d_{2},n}+W_{t}^{B}(\omega))\right),Y_{t}^{\ve,\d_{1},n}-Y_{t}^{\ve,\d_{2},n})_{H}\\
 & \le2(\div\phi^{\d_{1}}\left(\nabla(Y_{t}^{\ve,\d_{1},n}+W_{t}^{B}(\omega))\right)-\div\phi^{\d_{2}}\left(\nabla(Y_{t}^{\ve,\d_{2},n}+W_{t}^{B}(\omega))\right),Y_{t}^{\ve,\d_{1},n}-Y_{t}^{\ve,\d_{2},n})_{H}.
\end{align*}
Using \cite[Appendix A, (A.6)]{GT15} we conclude that
\begin{align*}
 & \frac{{d}}{dt}\|Y_{t}^{\ve,\d_{1},n}-Y_{t}^{\ve,\d_{2},n}\|_{H}^{2}\\
 & \le C(\d_{1}+\d_{2})(1+\|Y_{t}^{\ve,\d_{1},n}\|_{\dot{H}^{1}}^{2}+\|Y_{t}^{\ve,\d_{2},n}\|_{\dot{H}^{1}}^{2}+\|W_{t}^{B}(\omega)\|_{\dot{H}^{1}}^{2}).
\end{align*}
Hence,
\begin{align*}
 & \|Y_{t}^{\ve,\d_{1},n}-Y_{t}^{\ve,\d_{2},n}\|_{H}^{2}\\
 & \le C(\d_{1}+\d_{2})(1+\int_{0}^{t}\|Y_{r}^{\ve,\d_{1},n}\|_{\dot{H}^{1}}^{2}dr+\int_{0}^{t}\|Y_{r}^{\ve,\d_{2},n}\|_{\dot{H}^{1}}^{2}dr+\int_{0}^{t}\|W_{r}^{B}(\omega)\|_{\dot{H}^{1}}^{2}dr)
\end{align*}
which, using \eqref{eq:Y-energy-bound}, implies convergence of $Y^{\ve,\d,n}$
in $C([0,T];H)$ for $\d\to0$. 

\emph{Step 4:} The limits $\ve\to0$, $n\to\infty$ can be justified
precisely as in the proof of \cite[Theorem 3.1]{GT15} and the proof
can be concluded as in \cite[Theorem 3.1]{GT15}. Progressive measurability
of $(t,\o)\mapsto Y_{t}(\o)$ follows from the respective property
of $Y^{\ve,\d,n}$.
\end{proof}
\begin{rem}
\label{rem:transformation}

\begin{enumerate}
\item Let $X$ be the unique SVI solution to \eqref{eq:singular_p_laplace-1}
and $Y$ be the unique SVI solution to \eqref{eq:det_eq_transformed-2}.
Then $X=Y+W^{B}$, $\P$-a.s..
\item Let $X^{\ve}$ be the unique SVI solution to \eqref{eq:nonlocal_plp-1-1-2}
and $Y^{\ve}$ be the unique variational solution to \eqref{eq:det_eq_transformed-2}
given by \cite[Theorem 4.2.4]{PR07}. Then $X^{\ve}=Y^{\ve}+W^{B}$,
$\P$-a.s..
\end{enumerate}
\end{rem}
\begin{proof}
\emph{(i):} If $Y$ is the unique SVI solution to \eqref{eq:det_eq_transformed-2},
then $\bar{X}:=Y+W^{B}$ is progressively measurable and taking the
expectation in \eqref{eq:pathw_SVI_energy} yields $\bar{X}\in L^{2}([0,T]\times\O;H)$.
It is then easy to see that $\bar{X}$ is an SVI solution to \eqref{eq:singular_p_laplace-1}.
Thus, uniqueness of SVI solutions implies $X=\bar{X}$, $\P$-a.s..
\emph{(ii): }Follows by the same proof as (i). In order to prove that
$\bar{X}^{\ve}$ is an SVI solution to \eqref{eq:nonlocal_plp-1-1-2}
see the proof of Theorem \ref{thm:weakly-pointwise} below, in particular
\eqref{eq:Yn_SVI_1}, \eqref{eq:Yn_SVI_2}.
\end{proof}

\subsection{Convergence of solutions}
\begin{proof}[Proof of Theorem \ref{thm:weakly-pointwise}]
Let $\ve_{n}\to0$ and set $Y^{n}:=Y^{\ve_{n}}$, $J^{n}=J^{\ve_{n}}$.

\textit{Step 1:} Weak convergence in $L^{2}([0,T];H)$

By \cite[Theorem 4.2.4]{PR07} there is a unique variational solution
$Y^{n}\in C([0,T];H)$ to \eqref{eq:det_eq_transformed-1-1} with
respect to the trivial Gelfand triple $V=H=L^{2}$. Indeed, it is
easy to see that $A^{n}$ satisfies the required hemicontinuity and
monotonicity on $L^{2}$. Concerning coercivity, using \cite[Lemma 6.5]{AVMRT10},
we note that
\begin{align}
 & _{V^{*}}\<A^{n}(u),u\>_{V}\nonumber \\
 & =-\frac{1}{2}\int_{\mcO}\int_{\mcO}J^{n}(\z-\xi)\phi(u(\xi)-u(\z))(u(\xi)-u(\z))\,d\xi\,d\z\label{eq:An-coerc}\\
 & =-\frac{1}{2}\int_{\mcO}\int_{\mcO}J^{n}(\z-\xi)|u(\xi)-u(\z)|^{p}\,d\xi\,d\z\nonumber \\
 & =-\frac{1}{2}\|u\|_{V_{n}}^{p}\nonumber \\
 & \le0,\nonumber 
\end{align}
for all $u\in L^{2}$, where we have set $V_{n}:=L_{J^{n}}^{p}\cap L_{\operatorname{av}}^{p}$.
Moreover, using \cite[Lemma 6.5]{AVMRT10} and H\"{o}lder's inequality,
we note that 
\begin{align}
(A^{n}(u),v)_{H}= & \int_{\Ocal}\int_{\Ocal}J^{n}(\z-\xi)|u(\xi)-u(\z)|^{p-2}(u(\xi)-u(\z))(v(\xi)-v(\z))\,d\xi\,d\z\nonumber \\
\le & \left(\int_{\Ocal}\int_{\Ocal}J^{n}(\z-\xi)|u(\xi)-u(\z)|^{p}\,d\xi\,d\z\right)^{\frac{p-1}{p}}\label{eq:An-growth}\\
 & \left(\int_{\Ocal}\int_{\Ocal}J^{n}(\z-\xi)|v(\xi)-v(\z)|^{p}\,d\xi\,d\z\right)^{\frac{1}{p}}\nonumber \\
\le & \|u\|_{V_{n}}^{p-1}\|v\|_{V_{n}}\quad\forall u,v\in L^{2}.\nonumber 
\end{align}
It is easy to see that $Y^{n}$ is an SVI solution to \eqref{eq:det_eq_transformed-1-1}:
Indeed, by the chain-rule we have that
\begin{align}
\norm{Y_{t}^{n}-Z_{t}}_{H}^{2} & \le\norm{x_{0}^{n}-Z_{0}}_{H}^{2}+2\int_{0}^{t}\vp^{n}(Z_{r}+W_{r}^{B})\,dr-2\int_{0}^{t}\vp^{n}(Y_{r}^{n}+W_{r}^{B})\,dr\label{eq:Yn_SVI_1}\\
 & -\int_{0}^{t}(G_{r},Y_{r}^{n}-Z_{r})dr,\nonumber 
\end{align}
for all $Z\in W^{1,2}([0,T];H)$ and $G:=\frac{d}{dt}Z$. In particular,
choosing $Z\equiv0$ we obtain that 
\begin{align*}
\norm{Y_{t}^{n}}_{H}^{2}+2\int_{0}^{t}\vp^{n}(Y_{r}^{n}+W_{r}^{B})\,dr & \le\norm{x_{0}}_{H}^{2}+2\int_{0}^{t}\vp^{n}(W_{r}^{B})\,dr.
\end{align*}
By \cite[p. 24, first equation]{GT15} we have 
\[
\vp^{n}(v)\le C\|v\|_{\dot{W}^{1,p}}\le C(1+\|v\|_{\dot{H}^{1}}^{2})\quad\forall v\in H_{\operatorname{av}}^{1}.
\]
Hence,
\begin{align}
\norm{Y_{t}^{n}}_{H}^{2}+2\int_{0}^{t}\vp^{n}(Y_{r}^{n}+W_{r}^{B})\,dr & \lesssim\norm{x_{0}}_{H}^{2}+\int_{0}^{t}\|W_{r}^{B}\|_{\dot{H}^{1}}^{2}\,dr+1.\label{eq:Yn_SVI_2}
\end{align}
and we may extract a subsequence (again denoted by $Y^{n}$) such
that
\[
Y^{n}\rightharpoonup Y\quad\text{in }L^{2}([0,T];H).
\]
Using the Mosco convergence of $\vp^{n}\to\vp$ and $\limsup_{n\to\infty}\vp^{n}(u)\le\vp(u)$
(cf.~\cite[Proposition 5.2]{GT15}) and Mosco convergence of integral
functionals (cf.~\cite[Appendix B]{GT15}), it is easy to see that
$Y$ is an SVI solution to \eqref{eq:det_eq_transformed-2}. Since
SVI solutions to \eqref{eq:det_eq_transformed-2} are unique by Proposition
\ref{prop:det:SVI} this implies weak convergence of the whole sequence
$Y^{n}$ to $Y$ in $L^{2}([0,T];H)$.

\textit{Step 2: }By the chain-rule, \eqref{eq:An-coerc} and \eqref{eq:An-growth}
we have that
\begin{align*}
\norm{Y_{t}^{n}}_{H}^{2}= & \norm{x_{0}}_{H}^{2}+2\int_{0}^{t}(A^{n}(Y_{r}^{n}+W_{r}^{B}(\omega)),Y_{r}^{n})_{H}\,dr\\
= & \norm{x_{0}}_{H}^{2}+2\int_{0}^{t}(A^{n}(Y_{r}^{n}+W_{r}^{B}(\omega)),Y_{r}^{n}+W_{r}^{B}(\omega))_{H}\,dr\\
 & -2\int_{0}^{t}(A^{n}(Y_{r}^{n}+W_{r}^{B}(\omega)),W_{r}^{B}(\omega))_{H}\,dr\\
\le & \norm{x_{0}}_{H}^{2}-2\int_{0}^{t}\|Y_{r}^{n}+W_{r}^{B}(\omega)\|_{V_{n}}^{p}\,dr\\
 & +2\int_{0}^{t}\|Y_{r}^{n}+W_{r}^{B}(\omega)\|_{V_{n}}^{p-1}\|W_{r}^{B}(\omega)\|_{V_{n}}\,dr\\
\le & \norm{x_{0}}_{H}^{2}-\int_{0}^{t}\|Y_{r}^{n}+W_{r}^{B}(\omega)\|_{V_{n}}^{p}\,dr+C\int_{0}^{t}\|W_{r}^{B}(\omega)\|_{V_{n}}^{p}\,dr.
\end{align*}
Hence, using \eqref{eq:veps_norm_ineq},
\begin{equation}
\begin{split}\sup_{t\in[0,T]}\norm{Y_{t}^{n}}_{H}^{2}+\int_{0}^{T}\|Y_{r}^{n}+W_{r}^{B}(\omega)\|_{V_{n}}^{p}\,dr & \le\norm{x_{0}}_{H}^{2}+C\int_{0}^{T}\|W_{r}^{B}(\omega)\|_{H^{1}}^{p}\,dr.\end{split}
\label{eq:Yn-bound}
\end{equation}
We continue with an argument from \cite{G13-2}: Consider the set
\[
\Kcal:=\{(Y^{n},v)_{H}\;:\;v\in W^{1,p}\cap H,\;\norm{v}_{H}\vee\norm{v}_{W^{1,p}}\le1,\;n\in\N\}\subset C([0,T]).
\]
By \eqref{eq:Yn-bound}, $\Kcal$ is bounded in $C([0,T])$. Moreover,
by \eqref{eq:An-growth} and \eqref{eq:Yn-bound},
\begin{align*}
(Y_{t+s}^{n}-Y_{t}^{n},v)_{H} & =\int_{t}^{t+s}((Y^{n})_{r}^{\prime},v)_{H}\,dr\\
 & =\int_{t}^{t+s}(A^{n}(Y_{r}^{n}+W_{r}^{B}(\omega)),v)_{H}\,dr\\
 & \le\int_{t}^{t+s}\|Y_{r}^{n}+W_{r}^{B}(\omega)\|_{V_{n}}^{p-1}\|v\|_{V_{n}}\,dr\\
 & \le C{s}^{1/p}\int_{0}^{T}\|Y_{r}^{n}+W_{r}^{B}(\omega)\|_{V_{n}}^{p}dr\\
 & \le C{s}^{1/p},
\end{align*}
where we used \eqref{eq:veps_norm_ineq}. Hence, $\Kcal$ is a set
of equibounded, equicontinuous functions and thus is relatively compact
in $C([0,T])$ by the Arzel\`{a}-Ascoli theorem. Thus, for every $v\in W^{1,p}\cap H$,
$\norm{v}_{H}\vee\norm{v}_{W^{1,p}}\le1$, there is a subsequence
(again denoted by $Y^{n}$) such that $(Y^{n},v)_{H}\to g$ in $C([0,T])$.
Since also $(Y^{n},v)_{H}\to(Y,v)_{H}$ in $L^{2}([0,T])$ by step
one, we have $g=(Y,v)_{H}$. Hence, for each $v\in W^{1,p}\cap H$,
the whole sequence $(Y^{n},v)_{H}$ converges to $(Y,v)_{H}$ in $C([0,T])$. 

For $h\in H$, $\ve>0$ we can choose $v^{\ve}\in W^{1,p}\cap H$
such that $\|h-v^{\ve}\|\le\ve$. Then
\begin{align*}
(Y_{t}^{n}-Y_{t},h)_{H} & =(Y_{t}^{n}-Y_{t},h-v^{\ve})_{H}+(Y_{t}^{n}-Y_{t},v^{\ve})_{H}\\
 & \le\|Y_{t}^{n}-Y_{t}\|_{H}\|h-v^{\ve}\|_{H}+(Y_{t}^{n}-Y_{t},v^{\ve})_{H}\\
 & \le C\ve+(Y_{t}^{n}-Y_{t},v^{\ve})_{H}.
\end{align*}
Hence, choosing $n$ large enough implies
\begin{align*}
(Y_{t}^{n}-Y_{t},h)_{H} & \le C\ve\quad\forall n\ge n_{0}(\ve),
\end{align*}
that is $Y_{t}^{n}\rightharpoonup Y_{t}$ weakly in $H$ for $n\to\infty$. 

Since $Y^{n}=X^{n}-W^{B}$, $Y=X-W^{B}$, $\P$-a.s., this implies
weak convergence of $X_{t}^{n}$ to $X_{t}$ $\P$-a.s..

\textit{Step 3:} We next prove the convergence of the associated semigroups
$P_{t}^{\ve},P_{t}$. Let $F\in\mcF C_{b}^{1}(E)$ with $F=f(l_{1},\ldots,l_{k})$
and let $t\ge0$, $x\in H$. Further let $\ve_{n}\to0$ and set $P_{t}^{n}:=P_{t}^{\ve_{n}}$.
Then
\[
\begin{split}P_{t}^{n}F(x) & =\E F(X_{t}^{n,x})\\
 & =\E f(l_{1}(X_{t}^{n,x}),\ldots,l_{k}(X_{t}^{x,n}))\\
 & \to\E f(l_{1}(X_{t}^{x}),\ldots,l_{k}(X_{t}^{x}))\\
 & =\E F(X_{t}^{x})\\
 & =P_{t}F(x),
\end{split}
\]
as $n\to\infty$, by Lebesgue's dominated convergence theorem.
\end{proof}

\section{Convergence of invariant measures: Non-local to local\label{secc:convergence_measures}}

By Theorem \ref{thm:nonlocal_ergodicity}, for each $\ve>0$, there
exists a unique invariant measure $\mu^{\ve}$ to the stochastic nonlocal
$p$-Laplace equation \eqref{eq:nonlocal_plp} and by Theorem \ref{thm:local_ergodicity}
there is a unique invariant measure $\mu$ to the (local) stochastic
$p$-Laplace equation \eqref{eq:local_plp}. In this section we prove
weak$^{*}$ convergence of $\mu^{\ve}$ to $\mu$ in a suitable topology,
for $p\in(1,2)$.

Several difficulties appear, due to the nonlocal and singular-degenerate
nature of the SPDE \eqref{eq:nonlocal_plp}. First, we expect tightness
of $\mu^{\ve}$ on $H=L_{\operatorname{av}}^{2}$ only under stringent
dimensional restrictions. Indeed, for the (expected) limit $\mu$
we only know $\mu(W_{\operatorname{av}}^{1,p})=1$ which, roughly
speaking, would lead to assuming that the embedding $W_{\operatorname{av}}^{1,p}\hookrightarrow L_{\operatorname{av}}^{2}$
is compact and thus restrict to one spatial dimension, i.e.~$d=1$,
in general. Second, we only have weak convergence $X^{\ve}\rightharpoonup X$
in $H$. Therefore, the convergence of the associated semigroups $P_{t}^{n}F$
for general $F\in\Lip_{b}(H)$ is unclear; a crucial ingredient in
previously available methods.

The first problem is overcome in this section by considering weak
convergence of $\mu^{\ve}$ on $E=L_{\operatorname{av}}^{p}$ rather
than on $L_{\operatorname{av}}^{2}$. Again, for the limit $\mu$
we know $\mu(W_{\operatorname{av}}^{1,p})=1$. Hence, by compactness
of the embedding $W_{\operatorname{av}}^{1,p}\hookrightarrow L_{\operatorname{av}}^{p}$,
$\mu$ is concentrated on compact sets in $L_{\operatorname{av}}^{p}$
which suggests that tightness of $\mu^{\ve}$ on $L_{\operatorname{av}}^{p}$
should hold without restrictions on the dimension. Indeed, this is
established in Lemma \ref{prop:tightness} below. The resulting difficulty
of working with two topologies, weak$^{*}$ convergence of $\mu^{\ve}$
on $L_{\operatorname{av}}^{p}$ versus weak convergence of $X^{\ve}$
on $L_{\operatorname{av}}^{2}$ is solved in Lemma \ref{lem:weak-convergence-invariant-subsets}.
The second problem is overcome by first considering cylindrical functions
on $H$. For a cylindrical function $F$, weak convergence $X_{t}^{n}\rightharpoonup X_{t}$
in $H$ is enough to deduce $P_{t}^{n}F\to P_{t}F$. It turns out
that this is sufficient to deduce the weak$^{*}$ convergence $\mu^{\ve}\rightharpoonup^{*}\mu$
by means of a monotone class argument (cf.~proof of Theorem \ref{thm:convergence_ipm}).
The main result of this section is
\begin{thm}
\label{thm:convergence_ipm}Let $\mu^{\ve}$ be the unique invariant
measure to \eqref{eq:nonlocal_plp} and $\mu$ be the unique invariant
measure to \eqref{eq:local_plp}. Then $\mu^{\ve}\rightharpoonup^{*}\mu$
for $\ve\to0$ weakly$^{*}$ in the set of probability measures on
$L_{\operatorname{av}}^{p}(\Ocal)$, that is, for each bounded, Lipschitz
continuous function $F$ on $L_{\operatorname{av}}^{p}(\Ocal)$ we
have $\mu^{\ve}(F)\to\mu(F)$ for $\ve\to0$.
\end{thm}

\subsection{Asymptotic invariance}

In this section we provide a general result on the convergence of
invariant measures for convergent semigroups. Compared to previous
results \cite{CiotToe2} the main novelty here is to work with two
distinct topologies corresponding to the convergence of the invariant
measures on the one hand and to the convergence of the semigroups
on the other hand.
\begin{defn}
Let $E$ be a Banach space, $\mcG\subseteq\mcB_{b}(E)$ be a set of
bounded, measurable functions on $E$ and $P_{t}$ be a semigroup
on $E$. Then, a probability measure $\mu$ on $E$ is said to be
$\mcG$-invariant if
\[
\int_{E}P_{t}G\,d\mu=\int_{E}G\,d\mu\quad\forall G\in\mcG.
\]
\end{defn}
\begin{lem}
\label{lem:weak-convergence-invariant-subsets}Let $E,H$ be Banach
spaces with $H\hookrightarrow E$ dense. Further let $\mcG\subset\operatorname{Lip}_{b}(E)$,
$P_{t}^{n}$, $P_{t}$ be Feller semigroups on $H$ and $\mu^{n}$
be $\mcG$-invariant probability measures for $P_{t}^{n}$, for all
$n\in\N$. Suppose that $\mu_{n}\rightharpoonup^{*}\mu$ as probability
measures on $E$, the semigroups $P_{t}^{n}$ satisfy a uniform $e$-property,
that is, there exists a $C>0$ such that for all $F\in\operatorname{Lip}_{b}(E)$,
$x,y\in E$
\[
\lrabs{P_{t}^{n}F(x)-P_{t}^{n}F(y)}\le C\operatorname{Lip}(F)\norm{x-y}_{E}\quad\forall n\in\N,t\ge0
\]
and that for every $G\in\mcG$, $t\ge0$, $x\in H$,
\[
\lim_{n\to\infty}P_{t}^{n}G(x)=P_{t}G(x).
\]
Then $\mu$ is $\mcG$-invariant, i.e. 
\[
\int_{E}P_{t}G\,d\mu=\int_{E}G\,d\mu\quad\text{for all }G\in\mcG,t\ge0.
\]
\end{lem}
\begin{proof}
For two (Borel) probability measures $\nu_{1}$, $\nu_{2}$ on $(E,\mcB(E))$,
denote by $\beta_{E}(\nu_{1},\nu_{2})$ the bounded Lipschitz distance
between them, that is
\[
\beta_{E}(\nu_{1},\nu_{2}):=\sup\left\{ \lrabs{\int_{E}F\,d(\nu_{1}-\nu_{2})}\;:\;F\in\operatorname{Lip}_{b}(E),\;\norm{F}_{E,\infty}+\Lip_{E}(F)\le1\right\} .
\]
We have $ $$\operatorname{Lip}_{b}(E)\subseteq\operatorname{Lip}_{b}(H)$
and by continuous extension we can identify 
\[
\{F\in\operatorname{Lip}_{b}(H)\,:\,\exists C>0\ \text{s.t. }\|F(x)-F(y)\|_{E}\le C\|x-y\|_{E},\,\forall x,y\in E\}=\operatorname{Lip}_{b}(E).
\]
Accordingly, due to the $e$-property, $P_{t}^{n}:\operatorname{Lip}_{b}(E)\to\operatorname{Lip}_{b}(E)$.
Let $G\in\mcG$, $t\ge0$. We have that
\[
\begin{aligned} & \lrabs{\int_{E}G\,d\mu-\int_{E}P_{t}G\,d\mu}\\
 & \le\lrabs{\int_{E}G\,d\mu-\int_{E}P_{t}^{n}G\,d\mu_{n}}+\lrabs{\int_{E}P_{t}^{n}G\,d\mu_{n}-\int_{E}P_{t}^{n}G\,d\mu}\\
 & +\lrabs{\int_{E}P_{t}^{n}G\,d\mu-\int_{E}P_{t}G\,d\mu}.
\end{aligned}
\]
By the property of being $\mcG$-invariant measures, the first term
equals $\int_{E}G\,d(\mu-\mu_{n})$ and hence tends to zero as $n\to\infty$.
By the $e$-property, the second term can be bounded as follows (with
$\|F\|_{E,\infty}:=\sup_{x\in E}|F(x)|$)
\[
\beta_{E}(\mu_{n},\mu)\left[\norm{P_{t}^{n}G}_{E,\infty}+\operatorname{Lip}_{E}(P_{t}^{n}G)\right]\le\beta_{E}(\mu_{n},\mu_{0})\left[\norm{G}_{E,\infty}+C\operatorname{Lip}_{E}(G)\right]
\]
hence in turn tends to zero as $n\to\infty$ by weak convergence of
$\mu_{n}$ to $\mu$ and Lebesgue's dominated convergence, since (in
Polish spaces) the bounded Lipschitz metric generates the weak topology,
see e.g. \cite[1.12, pp. 73/74]{vdVW96}. Since $\mu^{n}(H)=1$ and
$\mu_{n}\rightharpoonup^{*}\mu$ we have $\mu(H)=1$. Thus, the third
term converges to zero by convergence of semigroups and Lebesgue's
dominated convergence theorem.
\end{proof}

\subsection{Tightness}

Below, taking complements of sets refers to the Polish space $E$,
that is, we denote $A^{c}:=E\setminus A$, for any set $A\subset E$.
\begin{defn}
A sequence of probability measures $\mu_{n}$ on a Polish space $E$
is called \emph{asymptotically tight}, if for each $\eta>0$ there
exists a compact set $K_{\eta}$ such that for each $\delta>0$ it
holds that
\[
\limsup_{n\to\infty}\mu_{n}((K_{\eta}^{\delta})^{c})<\eta,
\]
where $K_{\eta}^{\delta}\supset K_{\eta}$ is the open $\delta$-enlargement
of $K_{\eta}$.
\end{defn}
The next result can be found in \cite[Theorem 1.3.9]{vdVW96}.
\begin{lem}
If $\mu_{n}$ is asymptotically tight, then it is weakly relatively
compact.
\end{lem}
Let $\mu^{\ve}$, $\ve>0$ be the unique invariant measure associated
to \eqref{eq:nonlocal_plp}.
\begin{prop}
\label{prop:tightness}Let $\ve_{n}\searrow0$ as $n\to\infty$ and
set $\mu_{n}:=\mu_{\ve_{n}}$. Then $\mu_{n}$ is asymptotically tight
on $E:=L_{\operatorname{av}}^{p}(\Ocal)$.
\end{prop}
\begin{proof}
Let $\eta>0$ and $C:=\norm{B}_{L_{2}(H)}^{2}$. Recall $\vp_{\ve}\le K\vp$
for all $\ve\in(0,1]$ for some constant $K>0$. Then, by Poincar\'{e}'s
inequality, 
\[
K_{\eta}:=\left\{ x\in L_{\operatorname{av}}^{p}:\,\vp(x)\le\frac{2C}{\eta K}\right\} 
\]
 is a bounded set in $W_{\operatorname{av}}^{1,p}(\Ocal)$ and hence
compact in $L_{\operatorname{av}}^{p}(\Ocal)$. For $\d>0$, let $K_{\eta}^{\d}$
be the open $\d$-enlargement of $K_{\eta}$ in $E$. Let 
\[
G_{n}:=\left\{ x\in L_{\operatorname{av}}^{p}:\,\vp^{\ve_{n}}(x)\le\frac{2C}{\eta}\right\} 
\]
 for some $\ve_{n}\searrow0$, $n\to\infty$. Since $\vp_{\ve}\le K\vp$
for all $\ve\in(0,1]$, it holds that $G_{n}\supset K_{\eta}$ for
$n\in\N$. 

We claim that for each $\d>0$ there exists an $n_{0}\in\N$ such
that $G_{n}\subset K_{\eta}^{\d}$ for all $n\ge n_{0}$. We argue
by contradiction. If there exists $\d_{0}>0$, such that for all $n\in\N$
it holds that $G_{n}\not\subset K_{\eta}^{\d_{0}}$, then we can find
a sequence $x_{n}\in G_{n}\setminus K_{\eta}^{\d_{0}}$ such that
$\operatorname{dist}_{E}(x_{n},K_{\eta})\ge\d_{0}$ for every $n$.
By the definition of $G_{n}$ and \cite[Theorem 6.11 (2.)]{AVMRT10},
$\{x_{n}\}$ is relatively compact in $L_{\operatorname{av}}^{p}(\Ocal)$.
Hence, there exists a subsequence (denoted by $\{x_{n}\}$) such that
$x_{n}\to\bar{{x}}$ in $L_{\operatorname{av}}^{p}(\Ocal)$ and $\bar{x}\in W_{\operatorname{av}}^{1,p}(\Ocal)$.
By the Mosco convergence of $\vp_{\ve}\to\vp$ on $L^{p}$ (cf.~\cite[Proposition 5.2]{GT15})
we obtain that
\[
\vp(\bar{x})\le\liminf_{n\to\infty}\vp^{\ve_{n}}(x_{n})\le\frac{2C}{\eta},
\]
and thus $\bar{x}\in K_{\eta}$. Hence,
\[
\d_{0}\le\operatorname{dist}_{E}(x_{n},K_{\eta})\le\norm{\bar{x}-x_{n}}_{L_{\operatorname{av}}^{p}(\Ocal)}\underset{n\to\infty}{{\longrightarrow}}0,
\]
the desired contradiction.

Now, by Theorem \ref{thm:nonlocal_ergodicity}, for each $n\ge n_{0}(\d)$,
\[
\begin{aligned}\mu_{n}((K_{\eta}^{\d})^{c}) & \le\mu_{n}(G_{n}^{c})\le\frac{\eta}{2C}\int\vp^{\ve_{n}}(z)\,\mu_{n}(dz)\\
 & \le\frac{\eta}{2}<\eta.
\end{aligned}
\]
The proof is completed by taking the limsup as $n\to\infty$.
\end{proof}

\subsection{Proof of Theorem \ref{thm:convergence_ipm}}

We aim to apply Lemma \ref{lem:weak-convergence-invariant-subsets}
with $E=L_{\operatorname{av}}^{p}(\Ocal)$, $H=L_{\operatorname{av}}^{2}(\mcO)$.
Since $p\in(1,2)$, we have that $H\subseteq E$. Let $\mcG$ be the
space of cylindrical functions on $E$, that is, 
\[
\mcG=\mcF C_{b}^{1}(E).
\]
Let $\ve_{n}\to0$, set $\mu_{n}:=\mu_{\ve_{n}}$ and let $t\ge0$
be arbitrary, fixed. By Proposition \ref{prop:tightness} $\mu_{n}$
is asymptotically tight and thus has a weakly$^{*}$ convergent subsequence
(again denoted by $\mu_{n}$) such that $\mu_{n}\rightharpoonup\nu$.
The uniform \emph{$e$}-property for $P_{t}^{n}:=P_{t}^{\ve_{n}}$
on $E$ has been verified in Lemma \ref{lem:e-property} and by Theorem
\ref{thm:weakly-pointwise} we have
\[
P_{t}^{n}F(x)\to P_{t}F(x)\quad\text{for }n\to\infty
\]
for all $F\in\mcG$, $t\ge0$, $x\in H$. An application of Lemma
\ref{lem:weak-convergence-invariant-subsets} thus yields that $\nu$
is $\mcG$-invariant.

We show next that this implies that $\nu$ is an invariant measure
for $P_{t}$. First note that $\mcG$ is an algebra (w.r.t. pointwise
multiplication) of bounded real-valued functions on $E$ that contains
the constant functions. By \cite[II.3 a), p.54]{MR}, $\mcG$ separates
points of $E$, which by \cite[Theorem 6.8.9]{bogachev2007} implies
that $\mcG$ generates the Borel $\sigma$-algebra $\mcB(E)$. Set
\[
\mcH:=\mcH(\nu,t):=\left\{ F\in\mcB_{b}(E)\;:\;\int_{E}P_{t}F\,d\nu=\int_{E}F\,d\nu\right\} .
\]
Clearly, $1\in\mcH$ and $\mcH$ is closed under monotone convergence
by the Markov property and Beppo-Levi's monotone convergence lemma.
Further, $\mcH$ is closed under uniform convergence by Lebesgue's
dominated convergence theorem and the Markov property. We have already
shown $\mcG\subset\mcH$. Hence, by the monotone class theorem \cite[Theorem 2.12.9 (ii)]{bogachev2007},
$\mcB_{b}(E)\subset\mcH$ and therefore $\mcB_{b}(E)=\mcH$. Since
$\nu(H)=1$ this implies that $\nu$ is an invariant measure for $P_{t}$.
By Theorem \ref{thm:local_ergodicity} there is a unique invariant
measure $\mu$ for $P_{t}$. Thus $\mu=\nu$ and by uniqueness, the
whole sequence $\mu_{n}$ converges weakly$^{*}$ to $\mu$.

\appendix

\section{Notation\label{sec:Notation}}

We work with generic constants $C\ge0$, $c>0$ that are allowed to
change value from line to line and we write 
\[
A\lesssim B
\]
if there is a constant $C\ge0$ such that $A\le CB$. For a metric
space $(E,d)$, $R>0$, $x\in E$ we let $B_{R}(x)$ denote the open
ball of radius $R$ in $E$ centered at $x$. Moreover, we let $\mcB(E)$
denote the Borel sigma algebra and $\mcB_{b}(E)$ the space of bounded,
Borel-measurable functions on $E$. The $(d-1)$-dimensional unit
sphere in $\R^{d}$ is denoted by $S^{d-1}$. For notational convenience
we set 
\[
a^{[m]}:=|a|^{m-1}a\quad\text{for }a\in\R,m\ge0
\]
and
\[
|\text{\ensuremath{\xi}}|^{-1}\text{\ensuremath{\xi}}:=\begin{cases}
|\text{\ensuremath{\xi}}|^{-1}\text{\ensuremath{\xi}} & \text{if }\xi\in\R^{d}\setminus\{0\}\\
\bar{B}_{1}(0) & \text{if }\xi=0.
\end{cases}
\]

For $m\ge1$ we set $L^{m}(\mcO)$ to be the usual Lebesgue spaces
with norm $\|\cdot\|_{L^{m}}$ and we shall often use the shorthand
notation $L^{m}:=L^{m}(\mcO),\|\cdot\|_{m}:=\|\cdot\|_{L^{m}(\mcO)}$.
We let $B_{R}^{m}(x)$ be the open ball in $L^{m}$ of radius $R>0$
centered at $x\in L^{m}$. We further define $L_{\operatorname{av}}^{m}:=L_{\operatorname{av}}^{m}(\mcO)$
to be the space of all functions in $L^{m}$ with zero average, that
is,
\[
L_{\operatorname{av}}^{m}(\mcO):=\{v\in L^{m}(\mcO):\int_{\mcO}vd\xi=0\}
\]
and $H_{\operatorname{av}}^{k}:=H^{k}\cap L_{\operatorname{av}}^{2}$,
where $H^{k}$ are the usual Sobolev spaces. For a function $v\in L^{m}(\mcO)$
we define its extension to all of $\R^{d}$ by
\[
\bar{v}(\xi)=\begin{cases}
v(\xi) & \text{if }\xi\in\mcO\\
0 & \text{otherwise.}
\end{cases}
\]

Let $J:\R^{d}\to\R$ be a nonnegative, continuous, radial function
with compact support, $J(0)>0$, $\int_{\R^{d}}J(z)\,dz=1$. We then
consider the following nonlocal averaged Sobolev-type spaces: For
$\eps>0$, $m\ge1$, let $V_{\ve}:=L_{J^{\ve}}^{m}(\Ocal)$ be equal
to $L_{\operatorname{av}}^{m}(\Ocal)$ with the topology coming from
the norm
\begin{align*}
\norm{v}_{J^{\ve}}^{m} & :=\frac{C_{J,m}}{2m\eps^{d}}\int_{\Ocal}\int_{\Ocal}J\left(\frac{\xi-\zeta}{\eps}\right)\lrabs{\frac{v(\zeta)-v(\xi)}{\eps}}^{m}\,d\zeta\,d\xi\\
 & =\frac{C_{J,m}}{2m}\int_{\Ocal}\int_{\R^{d}}J(z)1_{\Ocal}(\xi+\eps z)\lrabs{\frac{\bar{v}(\xi+\eps z)-v(\xi)}{\eps}}^{m}\,dz\,d\xi,
\end{align*}
where $C_{J,m}$ is a normalization constant given by
\[
C_{J,m}^{-1}:=\frac{1}{2}\int_{\R^{d}}J(z)|z_{d}|^{m}\,dz.
\]
For notational convenience we set
\[
J^{\ve}(\xi):=\frac{C_{J,m}}{\eps^{d+m}}J\left(\frac{\xi}{\eps}\right)\quad\forall\xi\in\R^{d}.
\]
By \cite[Proposition 6.25]{AVMRT10} the norm $\norm{v}_{J^{\ve}}$
is equivalent to $\norm{v}_{m}$. In particular, $L_{J^{\ve}}^{m}(\Ocal)$
is a reflexive Banach space for all $m\in(1,\infty)$. Moreover, by
\cite{Brez2002},
\begin{equation}
\norm{\cdot}_{V_{\ve}}=\norm{\cdot}_{J^{\ve}}\le C\norm{\cdot}_{W^{1,m}},\label{eq:veps_norm_ineq}
\end{equation}
for some constant $C>0$ independent of $\ve>0$.

We say that a function $X\in L^{1}([0,T]\times\O;H)$ is $\F_{t}$-progressively
measurable if $X1_{[0,t]}\in L^{1}(\Bcal([0,t])\otimes\mcF_{t};H)$
for all $t\ge0$.

Let $E$ be a Banach space. For a Feller semigroup $P_{t}$ on $\mcB_{b}(E)$
we define the dual semigroup on the space of probability measures
$\mcM_{1}(E)$ on $E$ by
\[
P_{t}^{*}\mu(B):=\int_{E}P_{t}1_{B}(x)d\mu(x)
\]
and the time averages
\[
Q_{T}(x,B):=\frac{1}{T}\int_{0}^{T}P_{t}1_{B}(x)dt,\quad\forall B\in\mcB(E).
\]
We further set
\[
Q_{T}\mu(B):=\int_{E}Q_{T}(x,B)d\mu(x).
\]
We say that a probability measure $\mu$ on $E$ is invariant for
$P_{t}$ if $P_{t}^{*}\mu=\mu$ for all $t\ge0$. For an invariant
probability measure $\mu$ we define its basin of attraction by 
\[
\mcT(\mu):=\{x\in E:\,Q_{T}(x,\text{\textperiodcentered})=\frac{1}{T}\int_{0}^{T}P_{t}(x,\cdot)dt\rightharpoonup^{*}\mu\text{ for }T\to\infty\}\subset E.
\]
We say that $P_{t}$ satisfies the $e$-property if, for some constant
$C>0$,
\[
\|P_{t}F(x)-P_{t}F(y)\|_{E}\le C\Lip(F)\|x-y\|_{E}\quad\forall x,y\in\E,\,F\in\Lip(E).
\]

For a Banach space $E$ we define the space of cylindrical functions
on $E$ by 
\[
\mcF C_{b}^{1}(E):=\{f(l_{1},\ldots,l_{k})\,:\,k\in\N,\,l_{1},\ldots,l_{k}\in E^{\ast},\,f\in C_{b}^{1}(\R^{k})\}.
\]

\def\cprime{$'$}

\end{document}